  \documentclass[11pt]{article}

  \usepackage[letterpaper, left=1in, right=1in, top=1in, bottom=1in]{geometry}
  
  \usepackage[affil-it]{authblk}  
  \providecommand{\keywords}[1]{\textbf{Keywords: } #1}
  
  \usepackage{palatino}  
  \usepackage{setspace}  
  \usepackage{color}  
  \usepackage[dvipsnames]{xcolor}  
  
  \usepackage{enumitem}  

	\usepackage[labelfont=bf]{caption}	

  \usepackage{amsmath}  
  \usepackage{amssymb}  
  \usepackage{mathtools}  
  \usepackage{amsthm}  
  \usepackage{bm}  
  \usepackage{bbm}  
  \allowdisplaybreaks[4]  
	\usepackage{relsize}  
	\usepackage{algorithm}  
	\usepackage{algpseudocode}  
  
  \DeclareMathOperator*{\minimize}{minimize}
  
  \newcommand{\st}{\mathrm{subject\;to}}

	\theoremstyle{definition}  

	\newtheorem{proposition}{Proposition}

	\newtheorem{example}{Example}
	
	\newenvironment{remark}[1][Remark]{\begin{trivlist}
	\item[\hskip \labelsep {\bfseries #1}]}{\end{trivlist}}
	
  \usepackage{pifont}  

	\usepackage{longtable}  
	\usepackage{multirow}  
	\usepackage{array}  
	\usepackage{booktabs}  
	\usepackage{makecell}
    
  \usepackage{graphicx}  
  \usepackage{float}  
  \usepackage{subfig}  
  \graphicspath{{Figures/}}

	\usepackage{hyperref}
	\hypersetup{
		pdfstartview = {FitV},  
		colorlinks = true,    
		linkcolor = NavyBlue,  
		citecolor = NavyBlue,  
		filecolor = NavyBlue,  
		urlcolor = NavyBlue  
	}  

  \usepackage[numbers, sort&compress, super]{natbib}  

\title{A Unified Framework for Adjustable Robust Optimization with Endogenous Uncertainty}
\author[1]{Qi Zhang \thanks{Corresponding author (qizh@umn.edu)}}
\author[2]{Wei Feng}
\affil[1]{Department of Chemical Engineering and Materials Science, \break University of Minnesota, Minneapolis, MN 55455, USA}
\affil[2]{State Key Laboratory of Industrial Control Technology, College of Control Science and Engineering, Zhejiang University, Hangzhou 310027, China}
\date{}

\begin{document}

\maketitle

\begin{abstract}\normalsize
This work proposes a framework for multistage adjustable robust optimization that unifies the treatment of three different types of endogenous uncertainty, where decisions, respectively, (i) alter the uncertainty set, (ii) affect the materialization of uncertain parameters, and (iii) determine the time when the true values of uncertain parameters are observed. We provide a systematic analysis of the different types of endogenous uncertainty and highlight the connection between optimization under endogenous uncertainty and active learning. We consider decision-dependent polyhedral uncertainty sets and propose a decision rule approach that incorporates both continuous and binary recourse, including recourse decisions that affect the uncertainty set. The proposed method enables the modeling of decision-dependent nonanticipativity and results in a tractable reformulation of the problem. We demonstrate the effectiveness of the approach in computational experiments that cover a range of applications, including plant redesign, maintenance planning with inspections, optimizing revision points in capacity planning, and production scheduling with active parameter estimation. The results show significant benefits from the proper modeling of endogenous uncertainty and active learning.
\end{abstract}


\section{Introduction}
\label{sec:Intro}

Robust optimization is an approach to decision making under uncertainty that has gained tremendous popularity in the optimization community over the last two decades. It is particularly attractive in situations in which (i) the set of possible uncertainty realizations, commonly referred to as the uncertainty set, is known but the probability distribution is not, (ii) feasibility over the entire uncertainty set or the worst-case performance is of main interest, or (iii) alternative approaches, such as scenario-based stochastic programming, are computationally significantly less tractable. Earlier works on robust optimization \citep{Ben-Tal1998, Ghaoui1998} were limited to the \textit{static} case, which only considers \textit{here-and-now} decisions that need to be determined before the uncertainty is realized. In contrast, \textit{adjustable} robust optimization (ARO) \citep{Ben-Tal2004} provides a framework that also incorporates \textit{wait-and-see} (or recourse) decisions, which can be determined after the true values of some uncertain parameters are revealed. This makes ARO suitable for a large set of sequential decision-making problems, hence greatly expanding the applicability of the robust optimization paradigm. There has been a plethora of theoretical and practical contributions addressing a variety of ARO problems. A particular research focus has been the development of tractable approximations, e.g. in the form of decision rules \citep{Ben-Tal2004, Georghiou2019}, and solution methods \citep{Bertsimas2013, Zeng2013}, with many of the more recent efforts directed at the efficient handling of discrete recourse decisions \citep{Bertsimas2010a, Hanasusanto2015b, Bertsimas2015a, Postek2016, Bertsimas2018, Subramanyam2020}. We refer the reader to \citet{Yankoglu2019} for a comprehensive survey on ARO.

The term ``adjustable robust optimization'' had not appeared in the process systems engineering (PSE) literature until 2016 \citep{Lappas2016, Zhang2016e, Shi2016}. Only then, as formally discussed by \citet{Zhang2016d}, the PSE community realized that it actually has been applying the concept of (two-stage) ARO for decades, but under a different name: flexibility analysis \citep{Grossmann2014}. Many of the approaches developed in this line of work resemble the ones presented in the robust optimization literature. However, traditional flexibility analysis focuses primarily on nonlinear problems, which may be the reason why it so far has not been recognized or noticed by the operations research community. In the last few years, the number of ARO-related works in PSE has increased rapidly, addressing diverse applications in process design \citep{Yuan2018, Gong2018}, planning and scheduling \citep{Lappas2016, Zhang2016e, Shi2016, Ning2017, Lappas2018a, Feng2019, Rahal2020}, model predictive control \citep{Zhang2017e, Shang2019, Ning2020}, supply chain optimization \citep{Yue2016, Subramanyam2018}, etc.

In stochastic optimization, one distinguishes between \textit{exogenous} and \textit{endogenous} uncertainty. While exogenous uncertainty is not affected by the decision maker's actions, certain properties of endogenous uncertain parameters are decision-dependent. Endogenous uncertainty has its origin in stochastic programming \citep{Jonsbraten1998}, and mainly two types of endogenous uncertainty have been considered in the literature: \textit{type 1} where decisions alter the underlying probability distribution of an uncertain parameter \citep{Peeta2010, Escudero2018, Hellemo2018}, and \textit{type 2} where decisions affect the time at which an uncertain parameter materializes or its true value is revealed. The vast majority of existing works  address type-2 endogenous uncertainty \citep{Goel2004, Goel2006, Colvin2008, Colvin2010, Vayanos2011, Tarhan2013, Gupta2014, Christian2015, Hooshmand2016, Apap2017}. Here, the main challenge is the modeling of decision-dependent nonanticipativity, which is typically achieved by encoding a conditional scenario tree. The size of this scenario tree grows exponentially with the number of uncertain parameters and decisions affecting the uncertainty, which leads to dramatically increased computational complexity compared to the case with only exogenous uncertainty. Hence, most research efforts have focused on the development of more efficient formulations \citep{Goel2006, Colvin2008, Vayanos2011, Hooshmand2016, Apap2017} and solution methods \citep{Goel2006, Tarhan2013, Gupta2014, Colvin2010, Christian2015, Apap2017}. We refer to \citet{Apap2017} for a comprehensive review of stochastic programming with endogenous uncertainty.

Only recently, endogenous uncertainty has also been considered in robust optimization. \citet{Poss2013} and \citet{Nohadani2018} address static robust optimization with decision-dependent polyhedral uncertainty sets. A similar approach is taken by \citet{Lappas2018}, who further extend their framework to consider multistage robust process scheduling where the materialization of task-specific uncertainties, such as processing time and production yield, depends on the execution of the task \citep{Lappas2016, Lappas2018a}; however, the uncertainty set is only affected by first-stage decisions, and all recourse variables are continuous. \citet{Avraamidou2019} apply multiparametric programming to address endogenous uncertainty in two-stage robust optimization. Finally, \citet{Feng2020} consider the multistage case with both continuous and binary recourse as well as uncertainty sets that can be affected by decisions at every stage, for which a decision rule approach is applied to derive tractable approximations.

The above-mentioned references investigate problems in which decisions directly affect the stochastic nature of uncertain parameters, which in the robust optimization context translates into uncertainty sets whose shape, size, or dimensionality may be decision-dependent. Also, although the materialization of uncertain parameters can be decision-dependent, it is assumed that the true values of uncertain parameters are observed immediately after their materialization. In a series of recent works, Vayanos and coworkers \citep{Bertsimas2014a, Vayanos2020, Vayanos2020a} consider robust optimization with decision-dependent information discovery \citep{Vayanos2011}, where decisions determine whether and when uncertain parameters are observed. This constitutes a conceptually different type of endogenous uncertainty, where, instead of an underlying stochastic process, the source of uncertainty is the lack of information, which can be acquired with additional effort. In this case, the uncertainty set is reduced (``shrinks'') with additional observations. The proposed framework provides a means of modeling \textit{active learning}, which refers to selecting experiments sequentially based on the outcomes of previous experiments. It is hence a promising approach to optimizing the trade-off between \textit{exploration} and \textit{exploitation} in sequential decision-making problems, particularly those involving complex constraints. Existing works have applied this approach to the learning of customer behavior through dynamic pricing \citep{Bertsimas2014a} and active preference elicitation with application to housing allocation \citep{Vayanos2020a}.

We have come to realize that many problems in PSE require the simultaneous consideration of multiple types of endogenous uncertainty. Examples include strategic capacity planning, design of sensor networks, integrated online process optimization and parameter estimation, and maintenance planning with inspections. To address these problems, we need an approach that can consider all previously described types of endogenous uncertainty simultaneously, which, to the best of our knowledge, does not yet exist in the literature. In this work, we develop such a unified ARO framework, which significantly expands our capability to model data-driven optimization problems under uncertainty, particularly those involving active learning. The main contributions of this work are as follows:
\begin{itemize}
\item We introduce a new refined classification of endogenous uncertainty that more comprehensively reflects the current understanding of the subject. We further provide a unifying robust optimization perspective and highlight the connection between endogenous uncertainty and active learning.
\item We present two-stage and multistage robust optimization formulations and show that they can be used to simultaneously consider all defined types of endogenous uncertainty subject to polyhedral uncertainty sets.
\item A decision rule approach based on previous work \citep{Feng2020} is employed to develop tractable approximations of the given problems, which involve both continuous and binary recourse variables. One major novelty of this work is the use of auxiliary uncertain parameters to model decision-dependent nonanticipativity in the multistage setting. We derive a reformulation that can be directly solved using off-the-shelf optimization solvers.
\item The versatility and efficacy of the proposed modeling framework are demonstrated in several computational case studies, which include plant redesign, maintenance planning with inspections, optimizing revision points in capacity planning, and production scheduling with active parameter estimation.
\end{itemize}

The remainder of this paper is organized as follows. In Section \ref{sec:Classification}, we propose a new refined classification of endogenous uncertainty and systematically analyze the different types of endogenous uncertainty from a robust optimization perspective. In Section \ref{sec:ActiveLearning}, the connection between endogenous uncertainty and active learning is established. We develop the two-stage formulation incorporating all types of endogenous uncertainty in Section \ref{sec:TwoStage} before presenting the multistage formulation in Section \ref{sec:Multistage}. The proposed decision rule approach is detailed in Section \ref{sec:DecisionRules}. In Section \ref{sec:CaseStudies}, results from the computational studies are presented. Finally, in Section \ref{sec:Conclusions}, we close with some concluding remarks.

\paragraph{Notation} We use lowercase and uppercase boldface letters to denote vectors and matrices, respectively, e.g. $\bm{x} \in \mathbb{R}^n$ and $\bm{A} \in \mathbb{R}^{n \times m}$. Scalar quantities are denoted by non-boldface letters. A vector $\bm{x}$'s $i$th element is denoted by $x_i$. We use $\circ$ and $\mathbbm{1}(\cdot)$ to denote the Hadamard multiplication operator and the indicator function, respectively. Furthermore, $\bm{0}$, $\bm{e}$, and $\bm{E}$ are the zero vector, all-ones vector, and all-ones matrix, respectively, while $\bm{e}_i$ is the standard basis vector whose $i$th element is 1; their dimensions can be inferred from the context.

\section{Endogenous Uncertainty Through the Robust Optimization Lens}
\label{sec:Classification}

As mentioned in Section \ref{sec:Intro}, the literature distinguishes between type-1 and type-2 endogenous uncertainty. Here, we further refine the type-2 classification and propose the following definition of \textit{three} different types of endogenous uncertainty:
\begin{itemize}
\item Type 1: Decisions alter the probability distribution of uncertain parameters.
\item Type 2a: Decisions determine whether and when uncertain parameters materialize, i.e. actually become physically meaningful.
\item Type 2b: Decisions determine whether and when the true values of uncertain parameters are observed.
\end{itemize}
Endogenous uncertain parameters can be of more than one type. For example, a company may alter the probability distribution of the demand for its product by changing the selling price, in which case the product demand is a type-1 endogenous uncertain parameter. In addition, if we still have to decide whether or not to launch the product in the first place, the product demand is also type-2a endogenous since it will only materialize if the product is actually available on the market. Moreover, a customer's product demand is observed when the order is placed; it is type-2b endogenous if we can affect the time of the order, e.g. by specifying it in a contract.

In robust optimization, instead of specifying a probability distribution, uncertainty is described using an uncertainty set, which is the set of all realizations of the uncertain parameters considered possible in the problem. In other words, robust optimization only requires the support of the probability distribution. A two-stage robust optimization problem can be generally formulated as follows:
\begin{equation}
\label{eqn:TwoStageRO}
\begin{aligned}
  \minimize_{\bm{x} \in \mathcal{X}, \, \bm{\tilde{x}}} \quad & \max_{\bm{\xi} \in \Xi} f(\bm{x},\bm{\tilde{x}}(\bm{\xi}),\bm{\xi}) \\
  \st \quad & \bm{g}(\bm{x},\bm{\tilde{x}}(\bm{\xi}),\bm{\xi}) \leq \bm{0}, \; \bm{\tilde{x}}(\bm{\xi}) \in \widetilde{\mathcal{X}} \quad \forall \, \bm{\xi} \in \Xi,
\end{aligned}
\end{equation}
where $\bm{x} \in \mathcal{X}$ are the first-stage variables and $\bm{\tilde{x}} \in \widetilde{\mathcal{X}}$ are the second-stage variables, which are functions of the uncertain parameters $\bm{\xi} \in \mathbb{R}^K$. The objective and constraint functions are denoted by $f$ and $\bm{g}$, respectively. Problem \eqref{eqn:TwoStageRO} minimizes the worst-case cost while ensuring that the constraints are satisfied for all $\bm{\xi} \in \Xi$. By applying an epigraph reformulation, a problem of form \eqref{eqn:TwoStageRO} can be reformulated into a problem of the following form:
\begin{equation}
\label{eqn:TwoStageRO_epi}
\begin{aligned}
  \minimize_{\bm{x} \in \mathcal{X}, \, \bm{\tilde{x}}} \quad & x_1 \\
  \st \quad & \bm{g}(\bm{x},\bm{\tilde{x}}(\bm{\xi}),\bm{\xi}) \leq \bm{0}, \; \bm{\tilde{x}}(\bm{\xi}) \in \widetilde{\mathcal{X}} \quad \forall \, \bm{\xi} \in \Xi,
\end{aligned}
\end{equation}
where the objective function is deterministic. Formulation \eqref{eqn:TwoStageRO_epi} will be used in the remainder of this paper as it is more convenient for our discussion and subsequent reformulations.

\begin{remark}
There is a common misconception that a robust optimization problem has to consider a worst-case cost function. As it is apparent from \eqref{eqn:TwoStageRO_epi}, in theory, one can also consider any deterministic or deterministic equivalent (e.g. expressed in terms of scenarios) cost function. Constraint satisfaction with respect to the entire uncertainty set can be considered independently from the choice of objective function. Often, a non-worst-case objective function is more appropriate (see Section \ref{sec:CaseStudies} for two examples).
\end{remark}

In the case of endogenous uncertainty in the two-stage setting, the uncertainty depends on the first-stage decisions $\bm{x}$. For each type of endogenous uncertainty, we show how $\bm{x}$ affect the following three major components of problem \eqref{eqn:TwoStageRO_epi}:
\begin{itemize}
\item the uncertainty set $\Xi$, which is constant in the case of only exogenous uncertainty but depends on $\bm{x}$ if the uncertainty is endogenous, i.e. $\Xi = \Xi(\bm{x})$;
\item the \textit{deterministic feasible set} (DFS) defined as
\begin{equation*}
\mathcal{F}(\bm{x},\bm{\xi}) := \left\lbrace \bm{\tilde{x}} \in \widetilde{\mathcal{X}}: \bm{g}(\bm{x},\bm{\tilde{x}},\bm{\xi}) \leq \bm{0} \right\rbrace,
\end{equation*}
which is the set of feasible second-stage solutions given $\bm{x}$ and $\bm{\xi}$;
\item and the second-stage recourse decisions $\bm{\tilde{x}}(\bm{\xi})$ as functions of $\bm{\xi}$.
\end{itemize}

\paragraph{Type 1 (uncertainty alteration)} The probability distribution of type-1 endogenous uncertain parameters including its support can be altered by decisions. From the robust optimization perspective, this means that the shape and size of the uncertainty set $\Xi(\bm{x})$ depend on the first-stage decisions $\bm{x}$. For an example involving two uncertain parameters, $\xi_A$ and $\xi_B$, Figure \ref{fig:Type1} shows the uncertainty sets for two different choices of $\bm{x}$. One results in a polyhedral uncertainty set with four facets while the other results in a polytope with three facets.

\begin{figure}\centering
\subfloat[Type 1]{
  \label{fig:Type1}
	\fbox{\includegraphics[width=2in]{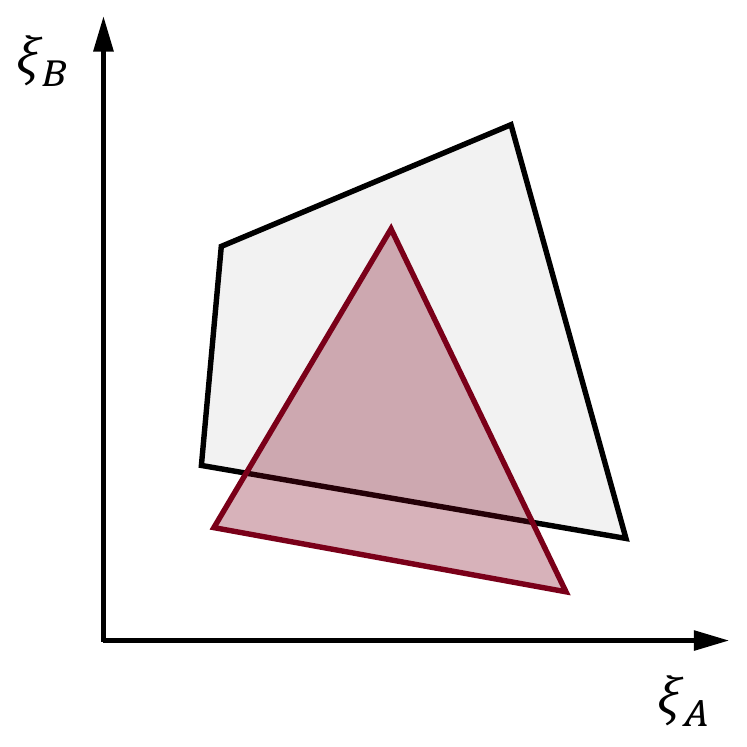}}} 
\subfloat[Type 2a]{
  \label{fig:Type2a}
	\fbox{\includegraphics[width=2in]{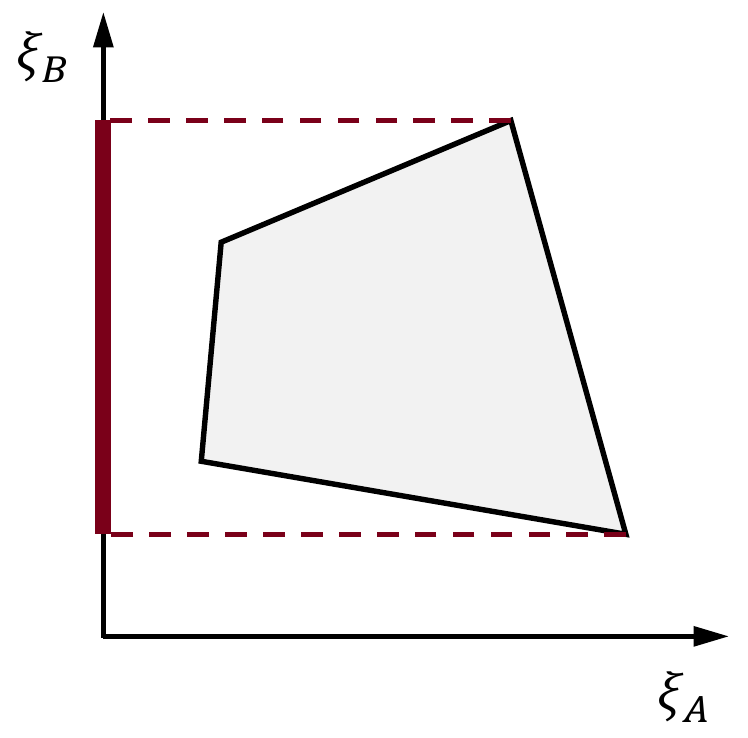}}} 
\subfloat[Type 2b]{
  \label{fig:Type2b}
	\fbox{\includegraphics[width=2in]{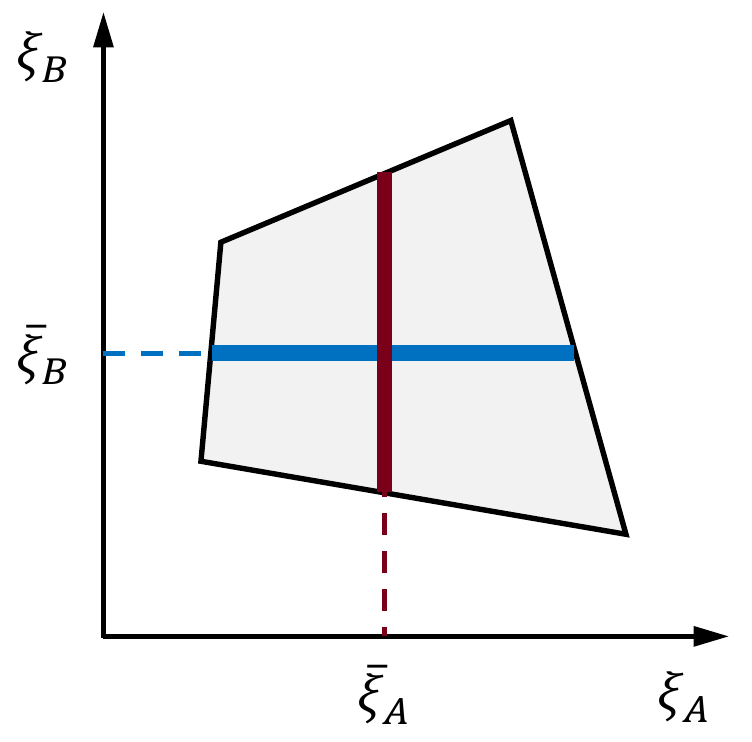}}}
\caption{Decisions affect the uncertainty set differently depending on the type of endogenous uncertainty: (a) type 1: shape and size of the uncertainty set may be altered; (b) type 2a: dimensionality of the uncertainty set changes depending on which uncertain parameters materialize; (c) type 2b: uncertainty set does not change, but its reduction over time does, i.e. it only reduces if realization of uncertainty is observed.}
\label{fig:UncertaintySets}
\end{figure}

\paragraph{Type 2a (uncertainty materialization)} In the case of type-2a endogenous uncertainty, $\bm{x}$ change the uncertainty set by affecting which uncertain parameters actually materialize. In the example shown in Figure \ref{fig:Type2a}, $\Xi(\bm{x})$ is the gray shaded polytope if both $\xi_A$ and $\xi_B$ materialize. However, if $\bm{x}$ are chosen such that $\xi_A$ does not materialize, $\Xi(\bm{x})$ becomes the solid red line segment that is a one-dimensional set for $\xi_B$; hence, $\bm{x}$ affect the dimensionality of the uncertainty set. Generally, when some uncertain parameters do not materialize, $\Xi(\bm{x})$ becomes the projection of the full-dimensional uncertainty set onto the space of the materialized uncertain parameters.

If an uncertain parameter does not materialize, it is physically meaningless and should therefore be irrelevant for the problem. More precisely, this means that the following DFS dependence condition has to hold: if $\bm{x}$ are chosen such that some uncertain parameters do not materialize, then $\mathcal{F}(\bm{x},\bm{\xi})$ has to be independent of those unmaterialized parameters. It is important to note that decision-dependent materialization of parameters is not a feature of uncertainty. It equally occurs in deterministic problems. Hence, the deterministic model, in which $\bm{\xi}$ is fixed, should already satisfy the DFS dependence condition, which usually directly extends to the robust formulation.

Finally, recourse decisions cannot depend on unmaterialized parameters. This, however, does not imply that recourse decisions are functions of all materialized parameters. While materialization is a necessary condition for an uncertain parameter to be considered in the recourse function, the sufficient condition is that its true value is observed. This subtle but important distinction is often ignored as almost all existing works on endogenous uncertainty assume that once an uncertain parameter materializes, its true value is automatically observed.

\paragraph{Type 2b (uncertainty observation)} In the case of type-2b endogenous uncertainty, we can decide if and when (in the multistage case) the true values of uncertain parameters are observed. Here, the underlying probability distribution and hence the uncertainty set is unaffected, but the reduction of uncertainty over time depends on when and which observations are made. This is illustrated in Figure \ref{fig:Type2b}, which shows an instance in which the true values of the uncertain parameters $\xi_A$ and $\xi_B$ are $\bar{\xi}_A$ and $\bar{\xi}_B$, respectively. If both parameters are observed, there will be no uncertainty in the second stage as the uncertainty set reduces to the point $(\bar{\xi}_A,\bar{\xi}_B)$. However, if we only observe $\xi_A$ and find it to be $\bar{\xi}_A$, the uncertainty with respect to $\xi_B$ remains and its reduced uncertainty set becomes the red solid line. Similarly, if we only observe $\xi_B$, the reduced uncertainty set for $\xi_A$ becomes the blue solid line. Note that the reduced uncertainty set is conditional, i.e. it depends on the observations made. Clearly, a parameter can only be observed if it materializes. Also, reiterating our previous remark, recourse decisions can only be functions of observed uncertain parameters.

The above discussion references the two-stage robust optimization problem but naturally extends to the multistage case. An overview of the characteristics of the different types of endogenous uncertainty is shown in Table \ref{tab:Classification}.

\begin{table}[ht]\centering
\setlength\tabcolsep{2pt}
	\caption{Overview of the impact of decisions on the uncertainty set, the dependence of the DFS on uncertain parameters, and the dependence of the recourse functions on uncertain parameters for the different types of endogenous uncertainty.}
	\begin{tabular}{ccccccccc}
	\toprule
	 	&&& \textbf{Uncertainty Set} && \textbf{DFS Dependence} && \textbf{Recourse Dependence} & \\ \midrule
	  & \textbf{Type 1}   && shape, size && unaffected && unaffected \\ \midrule
	  & \textbf{Type 2a}  && dimensionality && \makecell{only dependence on \\ materialized parameters} && \makecell{no dependence on \\ unmaterialized parameters} \\ \midrule
	  & \textbf{Type 2b}  && reduction over time && unaffected && \makecell{only dependence on \\ observed parameters} & \\
	\bottomrule
	\end{tabular}
	\label{tab:Classification}
\end{table}

\section{Endogenous Uncertainty and Active Learning}
\label{sec:ActiveLearning}

In the context of statistical learning, \textit{active learning} is the concept of actively gathering information for the purpose of statistical inference. It is also referred to as optimal experimental design as the decisions regarding information discovery are made with an objective in mind, such as minimizing the number of data points required to achieve a given level of confidence in the resulting statistical model. In this section, we highlight the connection between active learning and endogenous uncertainty, and show how it can be conceptualized in a set-based robust optimization framework.

The goal of learning is to gain insights into unknown properties, which can be viewed as uncertain parameters whose level of uncertainty we wish to reduce. Hence, learning is a means of uncertainty reduction. It does so by making new observations that \textit{condition} the probability distribution; it does not \textit{alter} the underlying distribution, which is an important distinction. With this perspective, we can model active learning using the notion of endogenous uncertainty. Let $\bm{\xi} \in \mathbb{R}^K$ be the uncertain parameters of interest and $\bm{\zeta} \in \mathbb{R}^M$ be uncertain parameters that can be observed before the realization of $\bm{\xi}$. Before observing any $\bm{\zeta}$, the probability distribution of $\bm{\xi}$ is $\mathbb{P}(\bm{\xi})$. The decision maker can now choose which $\bm{\zeta}$ to observe in order to improve the current information state with regard to $\bm{\xi}$, indicated by the binary variables $\bm{z} \in \{0,1\}^M$ where $z_i$ equals 1 if $\zeta_i$ is observed and 0 otherwise. After observing the chosen parameters, the updated probability distribution becomes $\mathbb{P}(\bm{\xi} \mid \zeta_i=\hat{\zeta}_i \; \forall \, i \text{ for which } z_i=1)$, where $\bm{\hat{\zeta}}$ denotes a specific observed realization of $\bm{\zeta}$. In a Bayesian sense, $\mathbb{P}(\bm{\xi})$ and $\mathbb{P}(\bm{\xi} | \bm{\zeta})$ can be interpreted as the prior and posterior distributions, respectively. Clearly, in this setting, $\bm{\zeta}$ are type-2b endogenous uncertain parameters. They may also be type-2a endogenous if their materialization is also decision-dependent. Note that $\bm{\xi}$ are considered exogenous uncertain parameters if their materialization and observation are not decision-dependent and if $\mathbb{P}(\bm{\xi})$ and $\mathbb{P}(\bm{\xi} | \bm{\zeta})$ are fixed.

In a robust optimization setting, uncertainty reduction implies reducing the uncertainty set. Using the same notation as above, uncertainty reduction through active learning can then be characterized as follows: Let $\Omega$ be the uncertainty set for the uncertain parameters $(\bm{\xi}, \bm{\zeta})$ before the observation of any $\bm{\zeta}$, i.e. $(\bm{\xi}, \bm{\zeta}) \in \Omega$. Once some $\bm{\zeta}$ are observed (indicated by the binary variables $\bm{z}$) with the realization $\bm{\hat{\zeta}}$, the uncertainty set reduces to $\bar{\Omega}(\bm{z}, \bm{\hat{\zeta}}) = \{(\bm{\xi}, \bm{\zeta}): (\bm{\xi}, \bm{\zeta}) \in \Xi, \; \bm{\zeta} = \bm{z} \circ \bm{\hat{\zeta}}\}$. Note that $\bar{\Omega}(\bm{z}, \bm{\hat{\zeta}}) \subseteq \Omega$ for every possible $\bm{\hat{\zeta}}$ as specified by $\Omega$.

\begin{example}
\label{exp:PilotPlant}
We consider investing in a large-scale industrial manufacturing plant, where the real plant capacity $q$ is unknown before it is actually built. To reduce the financial risk, we have the option of building a small-scale pilot plant whose achieved capacity $p$ can help better predict the capacity of the industrial-scale plant. The uncertainty set depicted in Figure \ref{fig:Pilot_plant} shows that $p$ and $q$ are highly correlated. If no pilot plant is built, the uncertain parameter $p$ does not materialize. As a result, we can only rely on our initial belief that the marginal uncertainty set for $q$ is $\mathcal{Q} = [q^{\min}, q^{\max}]$. However, if we do build a pilot plant and observe a pilot plant capacity of $\bar{p} \in [p^{\min}, p^{\max}]$, the uncertainty set for $q$ reduces to a significantly smaller $\bar{\mathcal{Q}}(\bar{p}) = \{q: \alpha + \beta \bar{p} - \Delta \leq q \leq \alpha + \beta \bar{p} + \Delta\}$. Note that in this example, $p$ is a type-2a and type-2b uncertain parameter since it only materializes (and can then be observed) if we decide to build the pilot plant.

\begin{figure}[ht!]\centering
\includegraphics[width=3.5in]{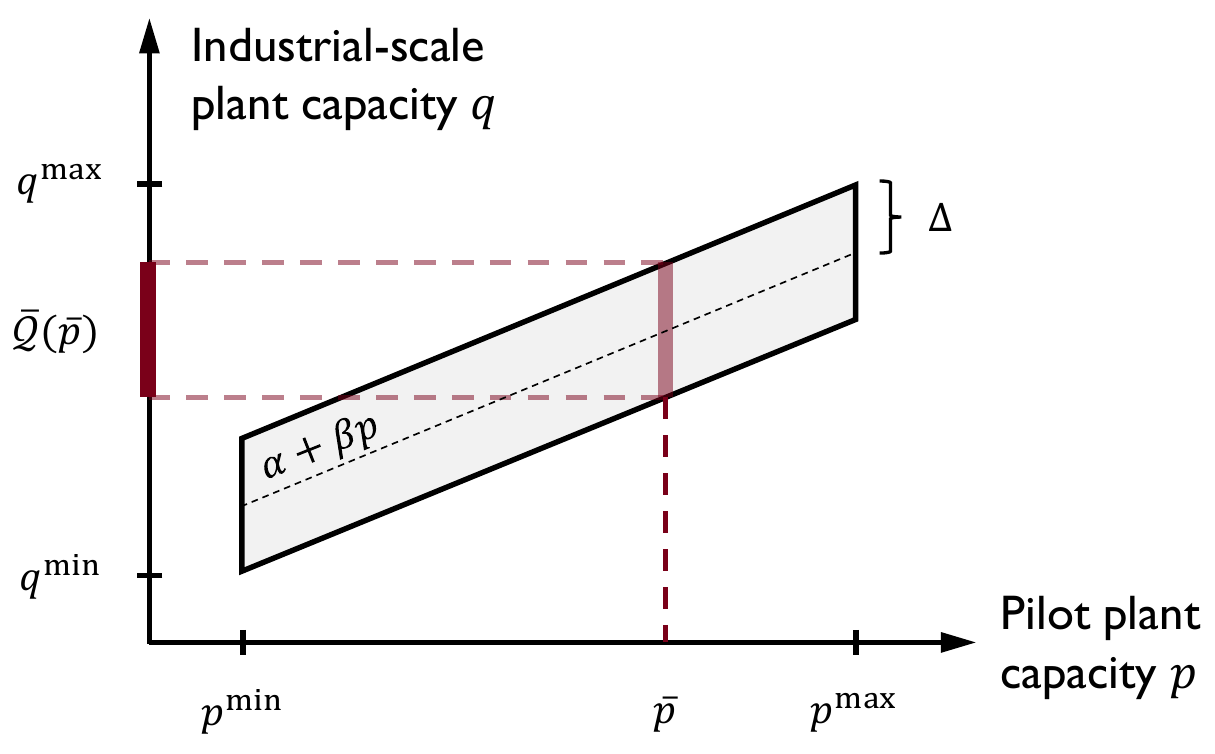}
\caption{The uncertainty set for the industrial-scale plant capacity is reduced by observing the capacity of a pilot plant.}
\label{fig:Pilot_plant}
\end{figure} 
\end{example}

In Example \ref{exp:PilotPlant}, we use the term \textit{marginal uncertainty set} and define it as the projection of the full-dimensional uncertainty set onto the space of a subset of variables, in this case $q$. Generally, given uncertain parameters $\bm{\xi}$ and $\bm{\zeta}$, observing $\bm{\zeta}$ further restricts the marginal uncertainty set for $\bm{\xi}$, denoted by $\Xi$, only if the uncertain parameters are correlated. The geometric interpretation is depicted in Figure \ref{fig:Cuts}, where an observation $\bm{\hat{\zeta}}$ results in cutting planes (indicated by the red dashed lines) added to $\Xi$ such that the updated marginal uncertainty set for $\bm{\xi}$ becomes $\overline{\Xi}(\bm{\hat{\zeta}})$.

\begin{figure}[ht!]\centering
\includegraphics[width=4.5in]{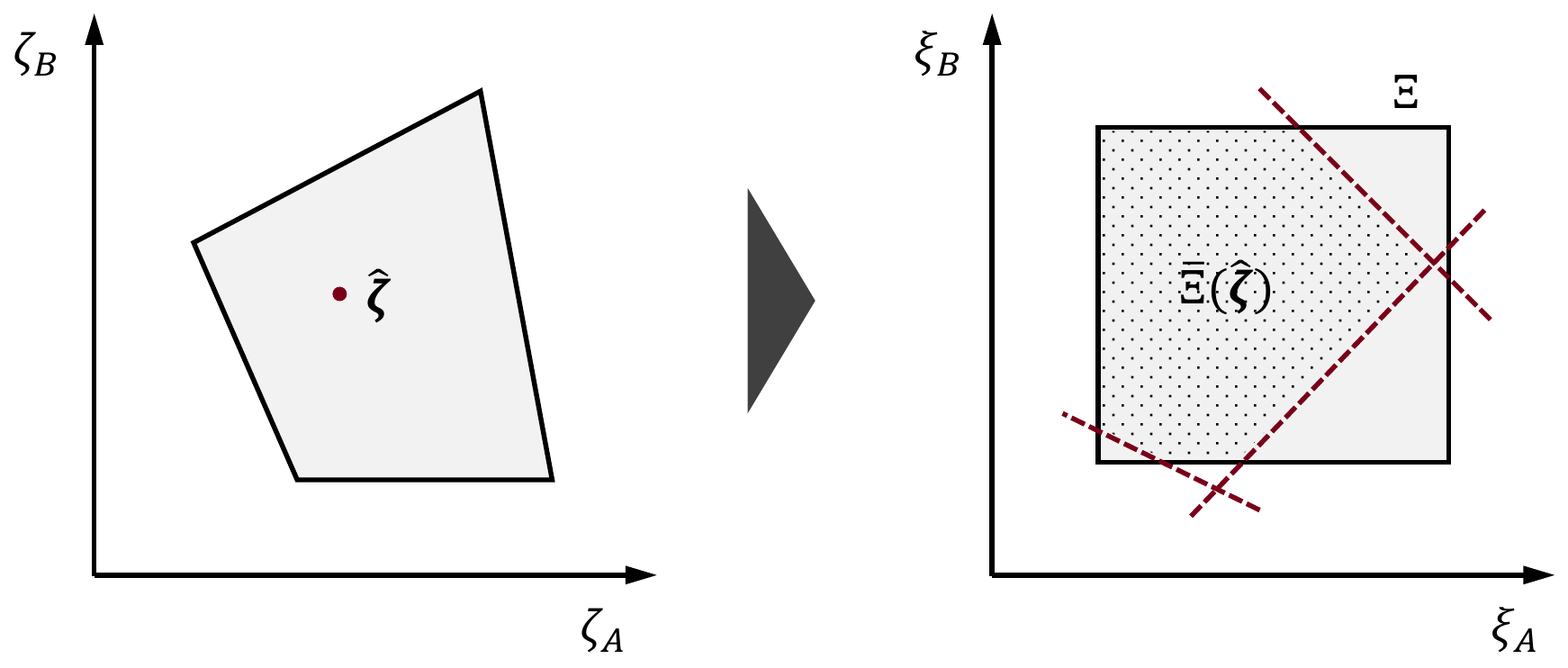}
\caption{The marginal uncertainty set for $\bm{\xi}$ is reduced from $\Xi$ (shaded rectangle) to $\overline{\Xi}(\bm{\hat{\zeta}})$ (pattern-filled polytope) by added cutting planes corresponding to observation $\bm{\hat{\zeta}}$.}
\label{fig:Cuts}
\end{figure} 

Modeling active learning as a robust optimization problem with endogenous uncertainty provides an opportunity for integrated optimization and learning, which are two tasks typically performed separately. Central to this problem is the trade-off between exploration and exploitation since our actions affect both the amount of new information and the reward that we receive. The proposed framework can readily incorporate this trade-off and is particularly well suited for problems with complex constraints and recourse decisions, which often cause difficulties for many alternative methods, such as Bayesian optimization \citep{Shahriari2016}, reinforcement learning \citep{Sutton2018}, and multi-armed bandit optimization \citep{Audibert2010}.

\section{Two-Stage Robust Optimization with Endogenous Uncertainty}
\label{sec:TwoStage}

In the following, we show how type-1, type-2a, and type-2b endogenous uncertainty can be incorporated in a two-stage robust optimization formulation, eventually leading to a unified model that can capture all these types of endogenous uncertainty simultaneously.

\subsection{Incorporating Type-1 Endogenous Uncertainty}

Consider the following two-stage robust optimization problem:
\begin{equation}
\label{eqn:TwoStage_type1}
\begin{aligned}
\minimize_{\bm{x}, \, \bm{\tilde{x}}} \quad & x_1 \\
\st \quad & \bm{x} \in \mathcal{X} \\
& \hspace{-8pt} \left.
\begin{array}{l}
  \bm{A}(\bm{\xi}) \bm{x} + \bm{\widetilde{A}}(\bm{\xi}) \bm{\tilde{x}}(\bm{\xi}) \leq \bm{b}(\bm{\xi}) \\[4pt]
  \bm{\tilde{x}}(\bm{\xi}) \in \widetilde{\mathcal{X}}
\end{array}
\right\rbrace \quad \forall \, \bm{\xi} \in \Xi(\bm{x}),
\end{aligned}
\end{equation}
where the dependence of the uncertainty set $\Xi$ on $\bm{x}$ indicates type-1 endogenous uncertainty. Throughout this work, we restrict our discussion to decision-dependent uncertainty sets of the following polyhedral form:
\begin{equation}
\label{eqn:UncertaintySet_type1}
\Xi(\bm{x}) = \left\lbrace \bm{\xi} \in \mathbb{R}^K: \bm{W} \bm{\xi} \leq \bm{U} \bm{x} \right\rbrace,
\end{equation}
where we assume that $\Xi(\bm{x})$ is always nonempty and bounded, i.e. $\Xi(\bm{x}) \subseteq \{\bm{\xi} \in \mathbb{R}^K: \bm{\xi}^{\min} \leq \bm{\xi} \leq \bm{\xi}^{\max}\}$ for any $\bm{x} \in \mathcal{X}$.

\begin{remark}
Some endogenous uncertainty can be modeled as exogenous uncertainty through a simple transformation. For example, consider a scalar endogenous uncertain parameter $\xi$ with the uncertainty set $\Xi(x) = \{\xi \in \mathbb{R}: 0 \leq \xi \leq x\}$, which depends on variable $x$. In this case, we could define a new uncertain parameter $\tilde{\xi}$ and add the equation $\xi = \tilde{\xi} x$ to the model. The uncertainty set for $\tilde{\xi}$ is then $\widetilde{\Xi} = \{\tilde{\xi} \in \mathbb{R}: 0 \leq \tilde{\xi} \leq 1\}$, which does not depend on $x$, and $\xi$ can now be treated as a variable that is adjustable with respect to $\tilde{\xi}$.
\end{remark}

\subsection{Incorporating Type-2a Endogenous Uncertainty}

To model type-2a endogenous uncertainty, we introduce binary variables $\bm{y} \in \{0,1\}^K$ to encode the materialization of uncertain parameters, i.e. $y_i = 1$ if and only if $\xi_i$ materializes. Then, both type-1 endogenous uncertainty with decision-dependent uncertainty set of the form \eqref{eqn:UncertaintySet_type1} and type-2a endogenous uncertainty can be captured in a two-stage formulation of the following form:
\begin{equation}
\label{eqn:TwoStage_type2a}
\begin{aligned}
\minimize_{\bm{x}, \, \bm{y}, \, \bm{\tilde{x}}} \quad & x_1 \\
\st \quad & \bm{x} \in \mathcal{X}, \; \bm{y} \in \mathcal{Y} \\
& \hspace{-8pt} \left.
\begin{array}{l}
  \bm{A}(\bm{\xi}) \bm{x} + \bm{D}(\bm{\xi}) \bm{y} + \bm{\widetilde{A}}(\bm{\xi}) \bm{\tilde{x}}(\bm{y} \circ \bm{\xi}) \leq \bm{b}(\bm{\xi}) \\[4pt]
  \bm{\tilde{x}}(\bm{y} \circ \bm{\xi}) \in \widetilde{\mathcal{X}}
\end{array}
\right\rbrace \quad \forall \, \bm{\xi} \in \Xi(\bm{x},\bm{y}),
\end{aligned}
\end{equation}
where $\mathcal{Y} \subseteq \{0,1\}^K$ and
\begin{equation}
\label{eqn:UncertaintySet_type2a}
\Xi(\bm{x},\bm{y}) = \left\lbrace \bm{\xi} \in \mathbb{R}^K: \bm{W} \bm{\xi} \leq \bm{U} \bm{x} + \bm{\overline{U}} \bm{y} \right\rbrace,
\end{equation}
assuming that $\Xi(\bm{x},\bm{y}) \subseteq \{\bm{\xi} \in \mathbb{R}^K: \bm{\xi}^{\min} \leq \bm{\xi} \leq \bm{\xi}^{\max}\}$ for any $\bm{x} \in \mathcal{X}$ and $\bm{y} \in \mathcal{Y}$. Note that formulation \eqref{eqn:TwoStage_type2a} is actually of the same form as \eqref{eqn:TwoStage_type1}; we have merely augmented the first-stage variables with $\bm{y}$ and restricted the recourse functions to only depend on materialized parameters, i.e. $\bm{y} \circ \bm{\xi}$. Here, the assumption is that all materialized uncertain parameters are also observed.

While the recourse dependence property of type-2a endogenous uncertainty is explicitly stated, it may not be immediately obvious how the change of dimensionality of the uncertainty set and the DFS dependence property can be captured in \eqref{eqn:TwoStage_type2a}. Here, the uncertainty set $\Xi(\bm{x},\bm{y})$ has to be constructed such that for a given $\bm{y}$, its projection onto the space of materialized uncertain parameters, i.e. $\xi_i$ for which $y_i=1$, forms the desired uncertainty set for the given $\bm{y}$. In addition, the constraints in \eqref{eqn:TwoStage_type2a} have to be constructed such that the feasible region is not affected by any $\xi_i \in \Xi_i(\bm{x},\bm{y})$ for which $y_i=0$, where $\Xi_i(\bm{x},\bm{y})$ denotes the projection of $\Xi(\bm{x},\bm{y})$ onto the $\xi_i$-space. Uncertainty sets and constraints satisfying these requirements can be constructed in various ways. To show the generality of formulation \eqref{eqn:TwoStage_type2a}, we provide the following result.

\begin{proposition}
\label{prp:TwoStage_type2a}
\textit{Formulation \eqref{eqn:TwoStage_type2a} is general in a sense that it can model the general case in which every possible set of materialized uncertain parameters results in a distinct set of constraints and a distinct $\bm{x}$-dependent uncertainty set of appropriate dimensionality.}
\end{proposition}

\begin{proof}
We provide a constructive proof for Proposition \ref{prp:TwoStage_type2a}. In the following, we refer to a possible set of materialized uncertain parameters as a \textit{materialization scenario}. Since $\bm{y} \in \{0,1\}^K$ is defined such that each materialization scenario is given by a specific $\bm{y}$ with $y_i=1$ if and only if the uncertain parameter $\xi_i$ materializes, the set of materialization scenarios is finite (maximum $2^K$). Let $S$ be the number of materialization scenarios and $\mathcal{S} := \{1, \dots, S\}$. We define a set $\mathcal{Y} := \{\bm{\bar{y}}^1, \dots, \bm{\bar{y}}^S\}$ where $\bm{\bar{y}}^s$ is the specific $\bm{y}$-vector that represents scenario $s$. In the most general case, each materialization scenario is associated with a distinct set of constraints and a distinct uncertainty set involving only the materialized uncertain parameters. Hence, the two-stage problem can be formulated as follows:
\begin{equation}
\label{eqn:TwoStage_type2a_general}
\begin{aligned}
\minimize_{\bm{x}, \, \bm{y}, \, \bm{\tilde{x}}} \quad & x_1 \\
\st \quad & \bm{x} \in \mathcal{X}, \; \bm{y} \in \mathcal{Y} \\
& \bigvee_{s \in \mathcal{S}} \:
\left[
\begin{array}{c}
	\bm{y} = \bm{\bar{y}}^s \\[4pt]
\hspace{-4pt} \left.
\begin{array}{c}
  \bm{A}^s(\bm{\xi}^s) \bm{x} + \bm{D}^s(\bm{\xi}^s) \bm{\bar{y}}^s + \bm{\widetilde{A}}^s(\bm{\xi}^s) \bm{\tilde{x}}(\bm{\xi}^s) \leq \bm{b}^s(\bm{\xi}^s) \\[4pt]
  \bm{\tilde{x}}(\bm{\xi}^s) \in \widetilde{\mathcal{X}}
\end{array}
\right\rbrace \; \forall \, \bm{\xi}^s \in \Xi^s(\bm{x})
\end{array}
\right],
\end{aligned}
\end{equation}
where the disjunction indicates that only the constraints for the selected materialization scenario need to be satisfied. Here, $\bm{\xi}^s$ only contains the $\xi_i$ for which $\bar{y}^s_i=1$ and is therefore a $K^s$-dimensional vector with $K^s=\bm{e}^{\top}\bm{\bar{y}}^s$. The $K^s$-dimensional scenario-specific uncertainty set is 
\begin{equation*}
\Xi^s(\bm{x}) = \left\lbrace \bm{\xi}^s \in \mathbb{R}^{K^s}: \bm{W}^s \bm{\xi}^s \leq \bm{U}^s \bm{x} \right\rbrace \subseteq \left\lbrace \bm{\xi}^s \in \mathbb{R}^{K^s}: \xi^{\min}_i \leq \xi_i \leq \xi^{\max}_i \;\; \forall \, i \text{ for which } \bar{y}^s_i=1 \right\rbrace.
\end{equation*}

Similar to the constraints, a disjunction can also be used to define an overall uncertainty set $\Xi(\bm{x},\bm{y})$. By doing so, problem \eqref{eqn:TwoStage_type2a_general} can be reformulated as
\begin{equation}
\label{eqn:TwoStage_type2a_general2}
\begin{aligned}
\minimize_{\bm{x}, \, \bm{y}, \, \bm{\tilde{x}}} \quad & x_1 \\
\st \quad & \bm{x} \in \mathcal{X}, \; \bm{y} \in \mathcal{Y} \\
& \bigvee_{s \in \mathcal{S}} \:
\left[
\begin{array}{c}
	\bm{y} = \bm{\bar{y}}^s \\[4pt]
	\bm{A}^s(\bm{\xi}^s) \bm{x} + \bm{D}^s(\bm{\xi}^s) \bm{\bar{y}}^s + \bm{\widetilde{A}}^s(\bm{\xi}^s) \bm{\tilde{x}}(\bm{\bar{y}}^s \circ \bm{\xi}) \leq \bm{b}^s(\bm{\xi}^s) \\[4pt]
	\bm{\tilde{x}}(\bm{\bar{y}}^s \circ \bm{\xi}) \in \widetilde{\mathcal{X}}
\end{array}
\right] \quad \forall \, \bm{\xi} \in \Xi(\bm{x},\bm{y})
\end{aligned}
\end{equation}
with
\begin{equation*}
\Xi(\bm{x},\bm{y}) = \left\lbrace
\bm{\xi} \in \mathbb{R}^K:
\bigvee_{s \in \mathcal{S}} \:
\left[
\begin{array}{c}
	\bm{y} = \bm{\bar{y}}^s \\[4pt]
	\bm{W}^s \bm{\xi}^s \leq \bm{U}^s \bm{x} \\[4pt]
	\xi_i = 0 \;\; \forall \, i \text{ for which } \bar{y}^s_i=0
\end{array}
\right]
\right\rbrace.
\end{equation*}
Note that in $\Xi(\bm{x},\bm{y})$, we set unmaterialized uncertain parameters to zero; this is a convenient but arbitrary choice since for a given materialization scenario $s$, only the projection of $\Xi(\bm{x},\bm{y})$ onto the $\bm{\xi}^s$-space matters.

With additional binary variables $\bm{\omega} \in \{0,1\}^S$, we can reformulate the disjunctions and arrive at the following formulation:
\begin{equation}
\label{eqn:TwoStage_type2a_general3}
\begin{aligned}
\minimize_{\bm{x}, \, \bm{\omega}, \, \bm{y}, \, \bm{\tilde{x}}} \quad & x_1 \\
\st \quad & \bm{x} \in \mathcal{X}, \; \bm{\omega} \in \{0,1\}^S, \; \bm{y} \in \mathcal{Y}  \\
& \sum_{s \in \mathcal{S}} \omega^s = 1, \quad \omega^s = 1 \Leftrightarrow \bm{y} = \bm{\bar{y}}^s \quad \forall \, s \in \mathcal{S} \\
& \bm{A}^s(\bm{\xi}^s) \bm{x} + \bm{D}^s(\bm{\xi}^s) \bm{y} + \bm{\widetilde{A}}^s(\bm{\xi}^s) \bm{\tilde{x}}(\bm{y} \circ \bm{\xi}) \leq \bm{b}^s(\bm{\xi}^s) + M^s \bm{e} (1-\omega^s) \quad \forall \, s \in \mathcal{S}, \, \bm{\xi} \in \Xi(\bm{x},\bm{\omega},\bm{y}) \\
& \bm{\tilde{x}}(\bm{y} \circ \bm{\xi}) \in \widetilde{\mathcal{X}} \quad \forall \, \bm{\xi} \in \Xi(\bm{x},\bm{\omega},\bm{y})
\end{aligned}
\end{equation}
with
\begin{equation*}
\Xi(\bm{x},\bm{\omega},\bm{y}) = \left\lbrace
\bm{\xi} \in \mathbb{R}^K:
\begin{array}{l}
	\bm{W}^s \bm{\xi}^s \leq \bm{U}^s \bm{x} + \widehat{M}^s \bm{e} (1-\omega^s) \quad \forall \, s \in \mathcal{S} \\[4pt]
	\bm{\xi}^{\min} \circ \bm{y} \leq \bm{\xi} \leq \bm{\xi}^{\max} \circ \bm{y}
\end{array}
\right\rbrace,
\end{equation*}
where $M^s$ and $\widehat{M}^s$ are vectors of sufficiently large parameters. As the statement $\omega^s = 1 \Leftrightarrow \bm{y} = \bm{\bar{y}}^s$ can be expressed as as set of linear constraints, e.g. $\bm{\bar{y}}^s \omega^s \leq \bm{y} \leq \bm{e} + (\bm{\bar{y}}^s-\bm{e}) \omega^s$, formulation \eqref{eqn:TwoStage_type2a_general3} can be easily transformed into the form of \eqref{eqn:TwoStage_type2a} by appropriate redefinition of variables, constraint matrices, and right-hand sides.
\end{proof}

\subsection{Incorporating Type-2b Endogenous Uncertainty}

Finally, we introduce additional binary variables $\bm{z} \in \{0,1\}^K$ to indicate which uncertain parameters are observed, and we arrive at the following unified two-stage robust optimization formulation that considers type-1, type-2a, and type-2b endogenous uncertainty:
\begin{equation}
\label{eqn:TwoStage_type2b}
\begin{aligned}
\minimize_{\bm{x}, \, \bm{y}, \, \bm{z}, \, \bm{\tilde{x}}} \quad & x_1 \\
\st \quad & \bm{x} \in \mathcal{X}, \; \bm{y} \in \mathcal{Y}, \; \bm{z} \in \mathcal{Z}, \; \bm{z} \leq \bm{y} \\
& \hspace{-8pt} \left.
\begin{array}{l}
  \bm{A}(\bm{\xi}) \bm{x} + \bm{D}(\bm{\xi}) \bm{y} + \bm{H}(\bm{\xi}) \bm{z} + \bm{\widetilde{A}}(\bm{\xi}) \bm{\tilde{x}}(\bm{z} \circ \bm{\xi}) \leq \bm{b}(\bm{\xi}) \\[4pt]
  \bm{\tilde{x}}(\bm{z} \circ \bm{\xi}) \in \widetilde{\mathcal{X}}
\end{array}
\right\rbrace \quad \forall \, \bm{\xi} \in \Xi(\bm{x},\bm{y}),
\end{aligned}
\end{equation}
where $\mathcal{Z} \subseteq \{0,1\}^K$. The recourse decisions $\bm{\tilde{x}}(\bm{z} \circ \bm{\xi})$ can only be functions of the observed parameters. Also, by requiring $\bm{z} \leq \bm{y}$, parameters can only be observed if they materialize. The formulation of the uncertainty set does not change from \eqref{eqn:UncertaintySet_type2a}:
\begin{equation*}
\Xi(\bm{x},\bm{y}) = \left\lbrace \bm{\xi} \in \mathbb{R}^K: \bm{W} \bm{\xi} \leq \bm{U} \bm{x} + \bm{\overline{U}} \bm{y} \right\rbrace.
\end{equation*}

\section{Multistage Robust Optimization with Endogenous Uncertainty}
\label{sec:Multistage}

In multistage robust optimization with endogenous uncertainty, decisions that affect the uncertainty can also be recourse variables. Moreover, while the two-stage problem can only consider decisions determining \textit{if} uncertain parameters materialize or are observed, multistage formulations can incorporate decisions affecting \textit{when}, i.e. in which stage, uncertain parameters materialize or are observed.

Let $\mathcal{T} = \{1,\dots,T\}$ denote the set of stages and $\mathcal{I}$ the set of uncertain parameters. An uncertain parameter $i \in \mathcal{I}$ can materialize anytime between (and including) given stages $\tau_i$ and $\bar{\tau}_i$. We define a vector of uncertain parameters $\bm{\xi}_t \in \mathbb{R}^{K_t}$ that contains all uncertain parameters $i$ for which $\tau_i=t$. We further define all uncertain parameters that can materialize in or prior to stage $t$ as $\bm{\xi}_{[t]} = (\bm{\xi}_1, \dots, \bm{\xi}_t)$, and $\bm{\xi} = \bm{\xi}_{[T]}$. The set of uncertain parameters for which $\tau_i=t$ is denoted by $\widetilde{\mathcal{I}}_t$, and we have $\widetilde{\mathcal{I}}_t \cap \widetilde{\mathcal{I}}_{t'} = \emptyset$ for all $t \neq t'$ and $\mathcal{I} = \bigcup_{t \in \mathcal{T}} \widetilde{\mathcal{I}}_t$. The set of uncertain parameters that can materialize in stage $t$ is denoted by $\bar{\mathcal{I}}_t$. In the special case where every uncertain parameter can only materialize in one stage, $\bar{\mathcal{I}}_t = \widetilde{\mathcal{I}}_t$. Similarly, let $\widehat{\mathcal{I}}_t$ denote the set of uncertain parameters that can be observed in stage $t$, and $\hat{\tau}_i$ and $\tilde{\tau}_i$ be the stages between which uncertain parameter $i$ can be observed. Since the observation of an uncertain parameter can only occur after its materialization, we have $\hat{\tau}_i \geq \tau_i$ and $\tilde{\tau}_i \geq \bar{\tau}_i$.

The general multistage robust optimization problem with type-1, type-2a, and type-2b endogenous uncertainty can be formulated as follows:
\begin{equation}
\label{eqn:Multistage}
\begin{aligned}
\minimize_{\bm{x}, \, \bm{y}, \, \bm{z}} \quad & x_{11} \\
\begin{array}{l}
  \st \; \vphantom{\sum\limits_{t'=1}^t} \\ [10pt]
  \vphantom{\bm{x}_t(\bm{\xi}_{[t]}) \in \mathcal{X}_t} \\ [4pt]
  \vphantom{\bm{y}_t(\bm{\xi}_{[t]}) \in \mathcal{Y}_t} \\ [4pt]
  \vphantom{\bm{z}_t(\bm{\xi}_{[t]}) \in \mathcal{Z}_t}
\end{array}
  & \hspace{-8pt} \left.
\begin{array}{l}
  \sum\limits_{t'=1}^t \left[ \bm{A}_{tt'}(\bm{\xi}_{[t]}) \bm{x}_{t'} + \bm{D}_{tt'}(\bm{\xi}_{[t]}) \bm{y}_{t'} + \bm{H}_{tt'}(\bm{\xi}_{[t]}) \bm{z}_{t'} \right] \leq \bm{b}_t(\bm{\xi}_{[t]}) \\ [10pt]
  \bm{x}_t = \bm{x}_t\left(\bm{\bar{z}}_{t-1} \circ \bm{\xi}_{[t]}\right), \; \bm{x}_t \in \mathcal{X}_t \\ [4pt]
  \bm{y}_t = \bm{y}_t\left(\bm{\bar{z}}_{t-1} \circ \bm{\xi}_{[t]}\right), \; \bm{y}_t \in \mathcal{Y}_t \\ [4pt]
  \bm{z}_t = \bm{z}_t\left(\bm{\bar{z}}_{t-1} \circ \bm{\xi}_{[t]}\right), \; \bm{z}_t \in \mathcal{Z}_t
\end{array}
\right\rbrace \\
  & \quad \quad \quad \quad \quad \quad \quad \quad \quad \quad \quad \quad \quad \quad \quad \quad \quad \quad \forall \, t \in \mathcal{T}, \, \bm{\xi}_{[t]} \in \Xi_t(\bm{x}_{[t-1]},\bm{y}_{[t-1]},\bm{z}_{[t-2]})
\end{aligned}
\end{equation}
where $\mathcal{Y}_t \in \{0,1\}^{|\bar{\mathcal{I}}_{t+1}|}$ and $\mathcal{Z}_t \in \{0,1\}^{|\widehat{\mathcal{I}}_{t+1}|}$. The notation is such that $\bm{x}_{[t]} = (\bm{x}_1,\dots,\bm{x}_t)$, $\bm{x} = \bm{x}_{[T]}$, $\bm{y}_{[t]} = (\bm{y}_1,\dots,\bm{y}_t)$, $\bm{y} = \bm{y}_{[T]}$, $\bm{z}_{[t]} = (\bm{z}_1,\dots,\bm{z}_t)$, and $\bm{z} = \bm{z}_{[T]}$. The binary variable $y_{ti}$ equals 1 if uncertain parameter $i$ materializes in stage $t+1$, whereas the binary variable $z_{ti}$ equals 1 if uncertain parameter $i$ is observed in stage $t+1$. We use auxiliary variables $\bm{\bar{z}}_{t-1} = \left(\sum_{t'=0}^{t-1} z_{t'1}, \dots, \sum_{t'=0}^{t-1} z_{t'K_{[t]}}\right)$ with $K_{[t]} = \sum_{t'=1}^t K_{t'}$ to indicate which uncertain parameters in $\bm{\xi}_{[t]}$ are observed up to stage $t$. The recourse variables in stage $t$, i.e. $\bm{x}_t$, $\bm{y}_t$, and $\bm{z}_t$, are functions of $\bm{\bar{z}}_{t-1} \circ \bm{\xi}_{[t]}$, which represent the uncertain parameters observed up to stage $t$. For notational convenience, we assume that $\bm{\xi}_1 = \xi_{11} = 1$, $\bm{y}_0 = y_{01} = 1$, and $\bm{z}_0 = z_{01} = 1$. The stage-specific uncertainty set $\Xi_t$ is defined such that $\Xi_1 = \{\bm{\xi}_1 \in \mathbb{R}: \xi_{11} = 1\}$, and for $t \geq 2$:
\begin{equation}
\label{eqn:MultistageUncertaintySet}
\Xi_t(\bm{x}_{[t-1]},\bm{y}_{[t-1]},\bm{z}_{[t-2]}) = \left\lbrace
\bm{\xi}_{[t]} \in \mathbb{R}^{K_{[t]}}:
\begin{array}{l}
  \bm{W}_t \bm{\xi}_{[t]} \leq \sum\limits_{t'=1}^{t-1} \left[ \bm{U}_{tt'} \bm{x}_{t'} + \bm{\overline{U}}_{tt'} \bm{y}_{t'} \right] \\ [8pt]
  \bm{x}_{t'} = \bm{x}_{t'}\left(\bm{\bar{z}}_{t-1} \circ \bm{\xi}_{[t']}\right) \quad \forall \, t'=1,\dots,t-1 \\ [4pt]
  \bm{y}_{t'} = \bm{y}_{t'}\left(\bm{\bar{z}}_{t-1} \circ \bm{\xi}_{[t']}\right) \quad \forall \, t'=1,\dots,t-1 \\ [4pt]
  \bm{z}_{t'} = \bm{z}_{t'}\left(\bm{\bar{z}}_{t-1} \circ \bm{\xi}_{[t']}\right) \quad \forall \, t'=1,\dots,t-1
\end{array}
\right\rbrace,
\end{equation}
where the polyhedral set directly depends on $\bm{x}_{[t-1]}$ and $\bm{y}_{[t-1]}$, which in turn are functions of observed uncertain parameters given by $\bm{z}_{[t-2]}$. We assume that $\Xi_t(\bm{x}_{[t-1]},\bm{y}_{[t-1]},\bm{z}_{[t-2]}) \subseteq \{\bm{\xi}_{[t]} \in \mathbb{R}^{K_{[t]}}: \bm{\xi}^{\min}_{[t]} \leq \bm{\xi}_{[t]} \leq \bm{\xi}^{\max}_{[t]}\}$ for all feasible $(\bm{x}_{[t-1]},\bm{y}_{[t-1]},\bm{z}_{[t-2]})$.

The linear constraints in problem \eqref{eqn:Multistage} include the following inequalities:
\begin{subequations}
\begin{align}
  & \sum_{t' = \tau_i}^{\bar{\tau}_i} y_{t'-1,i} \leq 1 \quad \forall \, i \in \mathcal{I} \label{eqn:Materialize} \\
  & \sum_{t' = \hat{\tau}_i}^{\tilde{\tau}_i} z_{t'-1,i} \leq 1 \quad \forall \, i \in \mathcal{I} \label{eqn:Observe} \\
  & \sum_{t' = \tau_i}^{\min \{t, \bar{\tau}_i\}} y_{t'-1,i} \geq z_{t-1,i} \quad \forall \, t \in \mathcal{T} \setminus \{1\}, \, i \in \widehat{\mathcal{I}}_t. \label{eqn:MaterializeThenObserve}
\end{align}
\end{subequations}
Constraints \eqref{eqn:Materialize} state that each uncertain parameter $i$ can only materialize between stages $\tau_i$ and $\bar{\tau}_i$; analogously, uncertain parameter $i$ can only be observed between stages $\hat{\tau}_i$ and $\tilde{\tau}_i$ according to \eqref{eqn:Observe}. Inequalities \eqref{eqn:MaterializeThenObserve} ensure that an uncertain parameter can only be observed after it has materialized.

\section{Decision Rule Approach}
\label{sec:DecisionRules}

In this section, we present a decision rule approach to (approximately) solve the general multistage robust optimization problem under endogenous uncertainty, which reduces to the two-stage case when $T=2$. The proposed approach is an extension of the method presented in \citet{Feng2020} and relies on the concept of lifted uncertainty \citep{Georghiou2015a}. Here, the main innovation is the use of auxiliary uncertain parameters to model decision-dependent nonanticipativity.

For ease of exposition and for tractability reasons, we consider the case of fixed recourse, i.e. $\bm{A}_{tt'}$, $\bm{D}_{tt'}$, and $\bm{H}_{tt'}$ do not depend on $\bm{\xi}_{[t]}$. In theory, the presented approach can be adapted to consider random recourse \citep{Bertsimas2018}, however, at considerable additional computational cost. We further assume that $\bm{b}_t(\bm{\xi}_{[t]}) = \bm{B}_t \bm{\xi}_{[t]}$ and $\mathcal{X}_t = \mathbb{R}^{\bar{N}_t} \times \{0,1\}^{\widehat{N}_t}$. Note that $\bm{x}_t$ contain both continuous and binary variables, and we define $\bm{x}_t = (\bm{x}^{\mathrm{c}}_t, \bm{x}^{\mathrm{b}}_t)$ such that $\bm{x}^{\mathrm{c}}_t \in \mathbb{R}^{\bar{N}_t}$ and $\bm{x}^{\mathrm{b}}_t \in \{0,1\}^{\widehat{N}_t}$. 

\subsection{Decision-Dependent Nonanticipativity}
\label{sec:Nonanticipativity}

In problem \eqref{eqn:Multistage}, recourse variables can only be adjusted based on the values of observed uncertain parameters; this is referred to as nonanticipativity. As a result, any decision rule applied to a recourse variable can only be a function of uncertain parameters that have been observed up to the corresponding stage. Suppose we apply decision rules with a linear structure such that, for example, $\bm{x}_t = \bm{\widetilde{X}}_{t} \bm{\xi}_{[t]}$. Then, intuitively, one may attempt to model nonanticipavity by incorporating additional constraints that force decision rule parameters corresponding to unobserved uncertain parameters to be zero, e.g. $-M \bar{z}_{t-1,i} \bm{e} \leq \bm{\tilde{x}}_{ti} \leq M \bar{z}_{t-1,i} \bm{e}$ with $M$ being a sufficiently large parameter. In the two-stage case, where $\bm{z}$ are only first-stage variables, this is a viable approach. However, in a multistage setting, $\bm{z}_{[t-1]}$ can themselves be adjustable recourse variables; in that case, the added constraints also require $\bm{\widetilde{X}}_{t}$ to be adjustable. Simply applying another decision rule to $\bm{\widetilde{X}}_{t}$ does not resolve the issue since the decision rule parameters for $\bm{\widetilde{X}}_{t}$ would also have to be adjustable because of their dependence on $\bm{z}_{[t-1]}$. For fixed uncertainty sets, \citet{Vayanos2011} propose an approach that divides the uncertainty set into a number of preselected subsets $\mathcal{S}$ and introduces recourse variables for each subset $s \in \mathcal{S}$. As a result, constant $\bm{z}_s$ can be selected for each subset $s$, which allows restrictions to be imposed on the decision rule parameters for the remaining recourse variables such that nonanticipativity is satisfied.  However, this approach is severely computationally intractable if the number of uncertainty subsets is large, which is only exacerbated in the case of problem \eqref{eqn:Multistage} where the uncertainty set is generally not fixed.

We propose an alternative, more tractable approach to incorporate decision-dependent nonanticipativity. Notice that given a linear decision rule structure, the dependence of a recourse variable on an uncertain parameter is equally eliminated if, instead of setting the corresponding decision rule coefficients to zero, the uncertain parameter itself is set to zero. In the special case where an uncertain parameter can only materialize in one stage and its materialization directly leads to its observation (i.e. $\bm{y} = \bm{z}$), this can be easily incorporated by including the following inequalities in the definition of the uncertainty set $\Xi_t$:
\begin{equation}
\label{eqn:NACineq}
\bm{\xi}^{\min}_{t'} \circ \bm{y}_{t'-1} \leq \bm{\xi}_{t'} \leq \bm{\xi}^{\max}_{t'} \circ \bm{y}_{t'-1} \quad \forall \, t'=1,\dots,t,
\end{equation}
which are valid since, as noted in the proof of Proposition \ref{prp:TwoStage_type2a}, the values of unmaterialized parameters can be set to zero without loss of generality. However, inequalities \eqref{eqn:NACineq} cannot capture the general case where there may be multiple stages in which an uncertain parameter can materialize or where despite being already materialized, one can still decide whether or not to observe the true value of the uncertain parameter in a given stage.

To address the general case, we first generalize \eqref{eqn:NACineq} such that uncertain parameters that can materialize in multiple stages are considered:
\begin{equation}
\label{eqn:NACineqGeneral}
\xi^{\min}_{\tau_i i} \sum_{t' = \tau_i}^{\min \{t, \bar{\tau}_i\}} y_{t'-1,i} \leq \xi_{\tau_i i} \leq \xi^{\max}_{\tau_i i} \sum_{t' = \tau_i}^{\min \{t, \bar{\tau}_i\}} y_{t'-1,i} \quad \forall \, i \in \bigcup_{t'=1}^t \widetilde{\mathcal{I}}_{t'}.
\end{equation}
We then introduce $\tilde{\tau}_i-\hat{\tau}_i+1$ auxiliary uncertain parameters $\zeta_{\hat{\tau}_i i}, \dots, \zeta_{\tilde{\tau}_i i}$ for every $i \in \mathcal{I}$. Note that for an original uncertain parameter $i$ for which $\tau_i = \tilde{\tau}_i$ and $y_{\tau_i-1,i} = z_{\tau_i-1,i}$ (which corresponds to the above-mentioned special case), we do not actually require auxiliary uncertain parameters; however, for ease of notation, we assume that we have auxiliary uncertain parameters associated with every $i \in \mathcal{I}$. We incorporate the following inequalities in the definition of the uncertainty set for stage $t$:
\begin{subequations}
\label{eqn:NACineqGeneral2}
\begin{align}
  & \xi_{\tau_i i} - \xi^{\max}_{\tau_i i} (1-z_{t'-1,i}) \leq \zeta_{t'i} \leq \xi_{\tau_i i} + M (1-z_{t'-1,i}) \quad \forall \, t'=1,\dots,t, \, i \in \widehat{\mathcal{I}}_{t'} \\
  & \xi^{\min}_{\tau_i i} z_{t'-1,i} \leq \zeta_{t'i} \leq \xi^{\max}_{\tau_i i} z_{t'-1,i} \quad \forall \, t'=1,\dots,t, \, i \in \widehat{\mathcal{I}}_{t'},
\end{align}
\end{subequations}
where $M$ is a sufficiently large big-M parameter. Inequalities \eqref{eqn:NACineqGeneral2} enforce that $\zeta_{t'i}$ takes the value of $\xi_{\tau_i i}$ if $\xi_{\tau_i i}$ is observed in stage $t'$, and is zero otherwise.

With a slight abuse of notation, we define the following new uncertainty set that considers the original and the auxiliary uncertain parameters:
\begin{equation*}
\begin{aligned}
  & \Theta_t(\bm{x}_{[t-1]},\bm{y}_{[t-1]},\bm{z}_{[t-1]}) = \left\lbrace
(\bm{\xi}_{[t]}, \bm{\zeta}_{[t]}) \in \mathbb{R}^{\sum_{t'=1}^t K_{t'}+|\widehat{\mathcal{I}}_{t'}|}: \sum_{t'=1}^t \left[ \bm{W}_{tt'} \bm{\xi}_{t'} + \bm{\overline{W}}_{tt'} \bm{\zeta}_{t'} \right] \right. \\
  & \quad \quad \quad \quad \quad \quad \quad \quad \quad \quad \quad \quad \leq \left. \sum\limits_{t'=1}^{t-1} \left[ \bm{U}^{\mathrm{c}}_{tt'} \bm{x}^{\mathrm{c}}_{t'}(\bm{\zeta}_{[t']}) + \bm{U}^{\mathrm{b}}_{tt'} \bm{x}^{\mathrm{b}}_{t'}(\bm{\zeta}_{[t']}) + \bm{\overline{U}}_{tt'} \bm{y}_{t'}(\bm{\zeta}_{[t']}) + \bm{\widehat{U}}_{tt'} \bm{z}_{t'}(\bm{\zeta}_{[t']}) \right] \right\rbrace,
\end{aligned}
\end{equation*}
where $\bm{W}_{tt'} \in \mathbb{R}^{Q_t \times K_{t'}}$, $\bm{\overline{W}}_{tt'} \in \mathbb{R}^{Q_t \times |\widehat{\mathcal{I}}_{t'}|}$, $\bm{U}^{\mathrm{c}}_{tt'} \in \mathbb{R}^{Q_t \times \bar{N}_{t'}}$, $\bm{U}^{\mathrm{b}}_{tt'} \in \mathbb{R}^{Q_t \times \widehat{N}_{t'}}$, $\bm{\overline{U}}_{tt'} \in \mathbb{R}^{Q_t \times |\bar{\mathcal{I}}_{t'+1}|}$, and $\bm{\widehat{U}}_{tt'} \in \mathbb{R}^{Q_t \times |\widehat{\mathcal{I}}_{t'+1}|}$ are chosen to incorporate the inequalities from the original uncertainty set \eqref{eqn:MultistageUncertaintySet} as well as inequalities \eqref{eqn:NACineqGeneral} and \eqref{eqn:NACineqGeneral2}. Notice that the recourse variables $\bm{x}^{\mathrm{c}}_t$, $\bm{x}^{\mathrm{b}}_t$, $\bm{y}_t$, and $\bm{z}_t$ now depend on $\bm{\zeta}_{[t]}$, which represent the uncertain parameters observed up to stage $t$.

\subsection{Lifted Uncertainty}

For each auxiliary uncertain parameter $\zeta_{ti}$, we introduce two sets of new uncertain parameters and hence create a higher-dimensional \textit{lifted} uncertainty space, which will facilitate the design of flexible decision rules. To this end, the bounded marginal support of $\zeta_{ti}$, which is the same as the one of $\xi_{\tau_i i}$, is partitioned into $r_i$ pieces defined by its lower and upper bounds as well as $r_i-1$ breakpoints, and we denote these points by $p_i^j$ with $j=0,1,\dots,r_i$ such that
\begin{equation*}
\xi_{\tau_i i}^{\min} = p_i^0 < p_i^1 < \cdots < p_i^{r_i - 1} < p_i^{r_i} = \xi_{\tau_i i}^{\max}.
\end{equation*}
We then define two lifting operators, $\bar{L}_t: \mathbb{R}^{|\widehat{\mathcal{I}}_t|} \mapsto \mathbb{R}^{\overline{K}_t}$ and $\widehat{L}_t: \mathbb{R}^{|\widehat{\mathcal{I}}_t|} \mapsto \{0,1\}^{\widehat{K}_t}$, for each $t \in \mathcal{T}$. The first lifting operator $\bar{L}_t(\bm{\zeta}_t)$ maps $\bm{\zeta}_t$ onto a $\overline{K}_t$-dimensional space with $\overline{K}_t = \sum_{i \in \widehat{\mathcal{I}}_t} r_i$ such that each lifted parameter is a piecewise linear function of the corresponding $\zeta_{ti}$ and is defined as follows:
\begin{equation*}
    \bar{L}_{ti}^j(\bm{\zeta}_t) := 
    \left \{
        \begin{array}{lcl}
         {\zeta_{ti}}  & \text{if} & r_i = 1 \\
         {\inf \left \{ \zeta_{ti}, p_i^j \right\} } & \text{if} & r_i > 1, \, j = 1 \\
         {\sup \left \{ \inf \left \{ \zeta_{ti}, p_i^j \right \} - p_i^{j-1}, 0 \right \}} & \text{if} & r_i > 1, \, j = 2, \ldots, r_i.
        \end{array}  
    \right.
\end{equation*}
By construction, we have $\zeta_{ti} = \sum_{j=1}^{r_i} \bar{L}_{ti}^j(\bm{\zeta}_t)$. The second lifting operator $\widehat{L}_t(\bm{\zeta}_t)$ maps $\bm{\zeta}_t$ onto a $\widehat{K}_t$-dimensional space with $\widehat{K}_t = \sum_{i \in \widehat{\mathcal{I}}_t} g_i$ and $g_i = \max\{ 1, \, r_i - 1 \}$. It is defined such that 
\begin{equation*}
    \widehat{L}_{ti}^j(\bm{\zeta}_t) := 
    \left \{
        \begin{array}{lcl}
         1  & \text{if} & r_i = 1 \\
         {\mathbbm{1}(\zeta_{ti} \geq p_i^j)} & \text{if} & r_i > 1, \, j = 1, \ldots, g_i,
        \end{array}
    \right .
\end{equation*}
which is a piecewise constant function of $\zeta_{ti}$.

Using the lifting operators, we can define the following lifted endogenous uncertainty set:
\begin{equation*}
\begin{aligned}
  & \widetilde{\Theta}'_t(\bm{x}_{[t-1]},\bm{y}_{[t-1]},\bm{z}_{[t-1]}) = \left\lbrace
(\bm{\xi}_{[t]}, \bm{\zeta}_{[t]}, \bm{\bar{\zeta}}_{[t]}, \bm{\hat{\zeta}}_{[t]}) \in \mathbb{R}^{\widetilde{K}_{[t]}}: \vphantom{\begin{array}{l} \sum\limits_{t'=1}^t \\ [8pt] \bm{\bar{\zeta}}_{t'} \end{array}} \right. \\
  & \quad \left.
\begin{array}{l}
  \sum\limits_{t'=1}^t \left[ \bm{W}_{tt'} \bm{\xi}_{t'} + \bm{\overline{W}}_{tt'} \bm{\zeta}_{t'} \right] \leq \sum\limits_{t'=1}^{t-1} \left[ \bm{U}^{\mathrm{c}}_{tt'} \bm{x}^{\mathrm{c}}_{t'}(\bm{\zeta}_{[t']}) + \bm{U}^{\mathrm{b}}_{tt'} \bm{x}^{\mathrm{b}}_{t'}(\bm{\zeta}_{[t']}) + \bm{\overline{U}}_{tt'} \bm{y}_{t'}(\bm{\zeta}_{[t']}) + \bm{\widehat{U}}_{tt'} \bm{z}_{t'}(\bm{\zeta}_{[t']}) \right] \\ [8pt]
  \bm{\bar{\zeta}}_{t'} = \bar{L}_{t'}\left(\bm{\zeta}_{t'} \right), \; \bm{\hat{\zeta}}_{t'}= \widehat{L}_{t'} \left( \bm{\zeta}_{t'} \right) \quad \forall \, t'=1,\dots,t
\end{array}
\right\rbrace,
\end{aligned}
\end{equation*}
where $\widetilde{K}_{[t]} = \sum_{t'=1}^t K_{t'}+|\widehat{\mathcal{I}}_{t'}|+\overline{K}_{t'}+\widehat{K}_{t'}$, and for all $i \in \widehat{\mathcal{I}}_t$, $\bm{\bar{\zeta}}_{ti} = (\bar{\zeta}_{ti}^1, \ldots, \bar{\zeta}_{ti}^{r_i})$, $ \bar{L}_{ti} = (\bar{L}_{ti}^1, \ldots, \bar{L}_{ti}^{r_i})$, $\bm{\hat{\zeta}}_{ti} = (\hat{\zeta}_{ti}^1, \ldots, \hat{\zeta}_{ti}^{g_i})$, and $\widehat{L}_{ti} = (\widehat{L}_{ti}^1, \ldots, \widehat{L}_{ti}^{g_i})$. The main purpose of defining such a lifted uncertainty set is to enable the construction of linear decision rules in the space of the lifted uncertain parameters. Notice that $\widetilde{\Theta}'_t$ is an open set due to the discontinuity in $\widehat{L}$. Now consider problem \eqref{eqn:Multistage} with $\Xi_t$ replaced by $\widetilde{\Theta}'_t$ and the recourse variables restricted to be linear functions of $\bm{\bar{\zeta}}$ and $\bm{\hat{\zeta}}$, and a second problem that is the same except that it considers the convex hull of the closure of $\widetilde{\Theta}'_t$, i.e. $\mathrm{conv}(\mathrm{cl}(\widetilde{\Theta}'_t))$, instead of $\widetilde{\Theta}'_t$. Following the same arguments as in \citet{Bertsimas2018}, one can show that these two problems are equivalent in a sense that they have the same optimal value and that there is a one-to-one mapping between their feasible and optimal solutions.

It is difficult to obtain a closed-form description of $\mathrm{conv}(\mathrm{cl}(\widetilde{\Theta}'_t))$ as it requires finding all vertices of $\mathrm{cl}(\widetilde{\Theta}'_t)$, which are decision-dependent. Moreover, the number of vertices grows exponentially with the dimensionality of $\mathrm{cl}(\widetilde{\Theta}'_t)$ such that including them would render most problems of practical relevance computationally intractable. Therefore, instead of $\mathrm{conv}(\mathrm{cl}(\widetilde{\Theta}'_t))$, we consider an outer approximation, which we construct as follows. For each $t \in \mathcal{T}$, $i \in \widehat{\mathcal{I}}_t$, and $j = 1,\dots,r_i$, we define a set of vertices $\widetilde{\mathcal{V}}_{ti}^j$ in the lifted space:
\begin{equation*}
\widetilde{\mathcal{V}}_{ti}^j := \left\lbrace
\bm{\tilde{v}} =  \left (v, \, \bm{\bar{v}}, \, \bm{\hat{v}} \right): 
\begin{array}{l}
  v \in \{p_i^{j-1}, p_i^j\} \\
  \bm{\bar{v}} = \bar{L}_{ti} (\bm{e}_{ti} v) \\ 
  \bm{\hat{v}} = \lim\limits_{\zeta_{ti} \rightarrow v, \, \zeta_{ti} \in [p_i^{j-1}, p_i^j]} \widehat{L}_{ti} (\bm{e}_{ti} \zeta_{ti}) 
\end{array}
\right\rbrace,
\end{equation*}
which allows the following convex hull formulation of the closure of the marginal support of $\bm{\tilde{\zeta}}_{ti}$:
\begin{equation*}
   Z_{ti} = \left \{
   \bm{\tilde{\zeta}}_{ti} = (\zeta_{ti}, \bm{\bar{\zeta}}_{ti}, \bm{\hat{\zeta}}_{ti}):
        \begin{array}{l}
             \sum\limits_{\bm{\tilde{v}} \in \widetilde{\mathcal{V}}_{ti}} \lambda(\bm{\tilde{v}}) = 1 \\
             \zeta_{ti} = \sum\limits_{\bm{\tilde{v}} \in \widetilde{\mathcal{V}}_{ti}} \lambda(\bm{\tilde{v}}) v \\
             \bm{\bar{\zeta}}_{ti} = \sum\limits_{\bm{\tilde{v}} \in \widetilde{\mathcal{V}}_{ti}} \lambda(\bm{\tilde{v}}) \bm{\bar{v}} \\
             \bm{\hat{\zeta}}_{ti} = \sum\limits_{\bm{\tilde{v}} \in \widetilde{\mathcal{V}}_{ti}} \lambda(\bm{\tilde{v}}) \bm{\hat{v}}  \\
             \lambda(\bm{\tilde{v}}) \in \mathbb{R}_+ \quad \forall \, \bm{\tilde{v}} \in \widetilde{\mathcal{V}}_{ti}
            \end{array}
            \right \},
\end{equation*}
where $\widetilde{\mathcal{V}}_{ti} = \bigcup_{j=1}^{r_i} \widetilde{\mathcal{V}}_{ti}^j$ and $\lambda(\bm{\tilde{v}})$ denotes the coefficient associated with a particular vertex $\bm{\tilde{v}} \in \widetilde{\mathcal{V}}_{ti}$. We can now obtain the following closed-form expression of an outer approximation of $\mathrm{conv}(\mathrm{cl}(\widetilde{\Xi}'_t))$:
\begin{equation}
\label{eqn:OuterApproximation}
\begin{aligned}
  & \left\lbrace \{\bm{\xi}_{[t]}\} \times\prod_{t'=1}^t \prod_{i \in \widehat{\mathcal{I}}_{t'}} Z_{t'i} \right\rbrace \bigcap 
\left\lbrace (\bm{\xi}_{[t]}, \bm{\zeta}_{[t]}, \bm{\bar{\zeta}}_{[t]}, \bm{\hat{\zeta}}_{[t]}) \in \mathbb{R}^{\widetilde{K}_{[t]}}: \sum\limits_{t'=1}^t \left[ \bm{W}_{tt'} \bm{\xi}_{t'} + \bm{\overline{W}}_{tt'} \bm{\zeta}_{t'} \right] \right. \\
  & \quad \quad \quad \quad \quad \quad \quad \quad \quad \quad \leq \left. \sum\limits_{t'=1}^{t-1} \left[ \bm{U}^{\mathrm{c}}_{tt'} \bm{x}^{\mathrm{c}}_{t'}(\bm{\zeta}_{[t']}) + \bm{U}^{\mathrm{b}}_{tt'} \bm{x}^{\mathrm{b}}_{t'}(\bm{\zeta}_{[t']}) + \bm{\overline{U}}_{tt'} \bm{y}_{t'}(\bm{\zeta}_{[t']}) + \bm{\widehat{U}}_{tt'} \bm{z}_{t'}(\bm{\zeta}_{[t']}) \right] \right\rbrace.
\end{aligned}
\end{equation}

\subsection{Decision Rule Approximation}

Following a similar approach as in Feng et al. (2020), we apply decision rules that are linear functions of the lifted uncertain parameters and parameterized as follows:
\begin{subequations}
\label{eqn:DecisionRules}
\begin{align}
  & \bm{x}^{\mathrm{c}}_{t} = \sum_{t'=1}^t \sum_{i \in \widehat{\mathcal{I}}_{t'}} \left( \bm{\overline{X}}^{\mathrm{c}}_{tt'i} \bm{\bar{\zeta}}_{t'i} + \bm{\widehat{X}}^{\mathrm{c}}_{tt'i} \bm{\hat{\zeta}}_{t'i} \right) \label{eqn:DecisionRulesCont} \\
  & \bm{x}^{\mathrm{b}}_t = \sum_{t'=1}^t \sum_{i \in \widehat{\mathcal{I}}_{t'}} \bm{\widehat{X}}^{\mathrm{b}}_{tt'i} \bm{\hat{\zeta}}_{t'i}, \quad \bm{y}_t = \sum_{t'=1}^t \sum_{i \in \widehat{\mathcal{I}}_{t'}} \bm{\widehat{Y}}_{tt'i} \bm{\hat{\zeta}}_{t'i}, \quad \bm{z}_t = \sum_{t'=1}^t \sum_{i \in \widehat{\mathcal{I}}_{t'}} \bm{\widehat{Z}}_{tt'i} \bm{\hat{\zeta}}_{t'i}, \label{eqn:DecisionRulesBin}
\end{align}
\end{subequations}
where $\bm{\overline{X}}^{\mathrm{c}}_{tt'i} \in \mathbb{R}^{\bar{N}_t \times r_i}$, $\bm{\widehat{X}}^{\mathrm{c}}_{tt'i} \in \mathbb{R}^{\bar{N}_t \times g_i}$, $\bm{\widehat{X}}^{\mathrm{b}}_{tt'i} \in \{-1, \,0, \, 1\}^{\widehat{N}_t \times g_i}$, $\bm{\widehat{Y}}_{tt'i} \in \{-1, \,0, \, 1\}^{|\bar{\mathcal{I}}_{t+1}| \times g_i}$, and $\bm{\widehat{Z}}_{tt'i} \in \{-1, \,0, \, 1\}^{|\widehat{\mathcal{I}}_{t+1}| \times g_i}$. The decision rules for the continuous variables in \eqref{eqn:DecisionRulesCont} can form continuous or discontinuous piecewise linear functions of $\bm{\zeta}_{[t]}$. With the given domains for $\bm{\widehat{X}}^{\mathrm{b}}_{tt'i}$, $\bm{\widehat{Y}}_{tt'i}$, and $\bm{\widehat{Z}}_{tt'i}$, a variable following a decision rule of the form \eqref{eqn:DecisionRulesBin} is guaranteed to be binary, i.e. can only take the value 0 or 1, if we also constrain it to be in $[0,1]$ \citep{Bertsimas2018}. This property implies that the integrality constraints on the binary variables can be relaxed when these decision rules are applied.

By applying the proposed decision rules in conjunction with the uncertainty set in \eqref{eqn:OuterApproximation}, decision-dependent nonanticipativity is captured, and we arrive at the following approximation of the multistage robust optimization problem under endogenous uncertainty:
\begin{equation}
\label{eqn:MultistageApprox}
\begin{aligned}
  \minimize_{\substack{\bm{\overline{X}}^{\mathrm{c}}, \, \bm{\widehat{X}}^{\mathrm{c}}, \\ \bm{\widehat{X}}^{\mathrm{b}}, \, \bm{\widehat{Y}}, \, \bm{\widehat{Z}}}} \quad & \bm{e}_1^{\top} \left(\bm{\overline{X}}^{\mathrm{c}}_{111} + \bm{\widehat{X}}^{\mathrm{c}}_{111} \right) \\
  \st \quad & \left(\bm{\overline{X}}^{\mathrm{c}}_{tt'i}, \bm{\widehat{X}}^{\mathrm{c}}_{tt'i}, \bm{\widehat{X}}^{\mathrm{b}}_{tt'i}, \bm{\widehat{Y}}_{tt'i}, \bm{\widehat{Z}}_{tt'i}\right) \in \Omega_{tt'i} \quad \forall \, t \in \mathcal{T}, \, t'=1,\dots,t, \, i \in \widehat{\mathcal{I}}_{t'} \\
  & \hspace{-8pt} \left.
\begin{array}{l}
  \sum\limits_{t'=1}^t \sum\limits_{t''=1}^{t'} \sum\limits_{i \in \widehat{\mathcal{I}}_{t''}} \left[ \bm{A}^{\mathrm{c}}_{tt'} \left( \bm{\overline{X}}^{\mathrm{c}}_{t't''i} \bm{\bar{\zeta}}_{t''i} + \bm{\widehat{X}}^{\mathrm{c}}_{t't''i} \bm{\hat{\zeta}}_{t''i} \right) + \bm{A}^{\mathrm{b}}_{tt'} \bm{\widehat{X}}^{\mathrm{b}}_{t't''i} \bm{\hat{\zeta}}_{t''i} \right. \\
  \quad \quad \quad \quad \quad \quad \quad + \left. \bm{D}_{tt'} \bm{\widehat{Y}}_{t't''i} \bm{\hat{\zeta}}_{t''i} + \bm{H}_{tt'} \bm{\widehat{Z}}_{t't''i} \bm{\hat{\zeta}}_{t''i} \right] \leq \bm{B}_t \bm{\xi}_{[t]} \\ [8pt]
  \bm{0} \leq \sum\limits_{t''=1}^{t'} \sum\limits_{i \in \widehat{\mathcal{I}}_{t''}} \bm{\widehat{X}}^{\mathrm{b}}_{t't''i} \bm{\hat{\zeta}}_{t''i} \leq \bm{e} \quad \forall \, t'=1,\dots,t \\ [4pt]
    \bm{0} \leq \sum\limits_{t''=1}^{t'} \sum\limits_{i \in \widehat{\mathcal{I}}_{t''}}  \bm{\widehat{Y}}_{t't''i} \bm{\hat{\zeta}}_{t''i} \leq \bm{e} \quad \forall \, t'=1,\dots,t \\ [4pt]
    \bm{0} \leq \sum\limits_{t''=1}^{t'} \sum\limits_{i \in \widehat{\mathcal{I}}_{t''}} \bm{\widehat{Z}}_{t't''i} \bm{\hat{\zeta}}_{t''i} \leq \bm{e} \quad \forall \, t'=1,\dots,t
\end{array}
\right\rbrace \\
  & \quad \quad \quad \quad \quad \forall \, t \in \mathcal{T}, \, (\bm{\xi}_{[t]}, \bm{\zeta}_{[t]}, \bm{\bar{\zeta}}_{[t]}, \bm{\hat{\zeta}}_{[t]}) \in \widetilde{\Theta}_t \left( \bm{\overline{X}}^{\mathrm{c}}_{[t-1]}, \bm{\widehat{X}}^{\mathrm{c}}_{[t-1]}, \bm{\widehat{X}}^{\mathrm{b}}_{[t-1]}, \bm{\widehat{Y}}_{[t-1]}, \bm{\widehat{Z}}_{[t-1]} \right),
\end{aligned}
\end{equation}
where
\begin{equation*}
\Omega_{tt'i} = \left \{
 \left(\bm{\overline{X}}^{\mathrm{c}}_{tt'i}, \bm{\widehat{X}}^{\mathrm{c}}_{tt'i}, \bm{\widehat{X}}^{\mathrm{b}}_{tt'i}, \bm{\widehat{Y}}_{tt'i}, \bm{\widehat{Z}}_{tt'i}\right):
  \begin{array}{l}
    \bm{\overline{X}}^{\mathrm{c}}_{tt'i} \in \mathbb{R}^{\bar{N}_t \times r_i} \\ [4pt]
    \bm{\widehat{X}}^{\mathrm{c}}_{tt'i} \in \mathbb{R}^{\bar{N}_t \times g_i} \\ [4pt]
    \bm{\widehat{X}}^{\mathrm{b}}_{tt'i} \in \{-1, \,0, \, 1\}^{\widehat{N}_t \times g_i} \\ [4pt]
    \bm{\widehat{Y}}_{tt'i} \in \{-1, \,0, \, 1\}^{|\bar{\mathcal{I}}_{t+1}| \times g_i} \\ [4pt]
    \bm{\widehat{Z}}_{tt'i} \in \{-1, \,0, \, 1\}^{|\widehat{\mathcal{I}}_{t+1}| \times g_i}
  \end{array}
  \right \},
\end{equation*}
and $\bm{A}^{\mathrm{c}}_{tt'}$ and $\bm{A}^{\mathrm{b}}_{tt'}$ are appropriately defined constraint matrices associated with the corresponding continuous and binary $\bm{x}$-variables, respectively. The uncertainty set is
\begin{equation*}
\begin{aligned}
  & \widetilde{\Theta}_t \left( \bm{\overline{X}}^{\mathrm{c}}_{[t-1]}, \bm{\widehat{X}}^{\mathrm{c}}_{[t-1]}, \bm{\widehat{X}}^{\mathrm{b}}_{[t-1]}, \bm{\widehat{Y}}_{[t-1]}, \bm{\widehat{Z}}_{[t-1]} \right) \\
  & \quad \quad \quad = \left\lbrace \{\bm{\xi}_{[t]}\} \times \prod_{t'=1}^t \prod_{i \in \widehat{\mathcal{I}}_{t'}} Z_{t'i} \right\rbrace \bigcap \, \widehat{\Theta}_t \left( \bm{\overline{X}}^{\mathrm{c}}_{[t-1]}, \bm{\widehat{X}}^{\mathrm{c}}_{[t-1]}, \bm{\widehat{X}}^{\mathrm{b}}_{[t-1]}, \bm{\widehat{Y}}_{[t-1]}, \bm{\widehat{Z}}_{[t-1]} \right)
\end{aligned}
\end{equation*}
with
\begin{equation*}\small
\begin{aligned}
  \widehat{\Theta}_t &= \left\lbrace (\bm{\xi}_{[t]}, \bm{\zeta}_{[t]}, \bm{\bar{\zeta}}_{[t]}, \bm{\hat{\zeta}}_{[t]}) \in \mathbb{R}^{\widetilde{K}_{[t]}}: \sum_{t'=1}^t \left[ \bm{W}_{tt'} \bm{\xi}_{t'} + \bm{\overline{W}}_{tt'} \bm{\zeta}_{t'} \right] \right. \\
  & \quad \quad \leq \left. \sum\limits_{t'=1}^{t-1} \sum\limits_{t''=1}^{t'} \sum\limits_{i \in \widehat{\mathcal{I}}_{t''}} \left[ \bm{U}^{\mathrm{c}}_{tt'} \left( \bm{\overline{X}}^{\mathrm{c}}_{t't''i} \bm{\bar{\zeta}}_{t''i} + \bm{\widehat{X}}^{\mathrm{c}}_{t't''i} \bm{\hat{\zeta}}_{t''i} \right) + \bm{U}^{\mathrm{b}}_{tt'} \bm{\widehat{X}}^{\mathrm{b}}_{t't''i} \bm{\hat{\zeta}}_{t''i} + \bm{\overline{U}}_{tt'} \bm{\widehat{Y}}_{t't''i} \bm{\hat{\zeta}}_{t''i} + \bm{\widehat{U}}_{tt'} \bm{\widehat{Z}}_{t't''i} \bm{\hat{\zeta}}_{t''i} \right] \right\rbrace.
\end{aligned}
\end{equation*}

For ease of exposition, we define new constraint and right-hand-side matrices such that we can express problem \eqref{eqn:MultistageApprox} in the following compact form:
\begin{equation}\small
\label{eqn:MultistageApproxCompact}
\begin{aligned}
  \minimize_{\substack{\bm{\overline{X}}^{\mathrm{c}}, \, \bm{\widehat{X}}^{\mathrm{c}}, \\ \bm{\widehat{X}}^{\mathrm{b}}, \, \bm{\widehat{Y}}, \, \bm{\widehat{Z}}}} \quad & \bm{e}_1^{\top} \left(\bm{\overline{X}}^{\mathrm{c}}_{111} + \bm{\widehat{X}}^{\mathrm{c}}_{111} \right) \\
  \st \quad & \sum\limits_{t'=1}^t \sum\limits_{t''=1}^{t'} \sum\limits_{i \in \widehat{\mathcal{I}}_{t''}} \left[ \bm{\widetilde{A}}^{\mathrm{c}}_{tt'} \bm{\overline{X}}^{\mathrm{c}}_{t't''i} \bm{\bar{\zeta}}_{t''i} + \left(\bm{\widetilde{A}}^{\mathrm{c}}_{tt'} \bm{\widehat{X}}^{\mathrm{c}}_{t't''i} + \bm{\widetilde{A}}^{\mathrm{b}}_{tt'} \bm{\widehat{X}}^{\mathrm{b}}_{t't''i} + \bm{\widetilde{D}}_{tt'} \bm{\widehat{Y}}_{t't''i} + \bm{\widetilde{H}}_{tt'} \bm{\widehat{Z}}_{t't''i} \right) \bm{\hat{\zeta}}_{t''i} \right] \\
  & \quad \quad \quad \leq \sum_{t'=1}^t \bm{\widetilde{B}}_{tt'} \bm{\xi}_{t'} \quad \forall \, t \in \mathcal{T}, \, (\bm{\xi}_{[t]}, \bm{\zeta}_{[t]}, \bm{\bar{\zeta}}_{[t]}, \bm{\hat{\zeta}}_{[t]}) \in \widetilde{\Theta}_t \left( \bm{\overline{X}}^{\mathrm{c}}_{[t-1]}, \bm{\widehat{X}}^{\mathrm{c}}_{[t-1]}, \bm{\widehat{X}}^{\mathrm{b}}_{[t-1]}, \bm{\widehat{Y}}_{[t-1]}, \bm{\widehat{Z}}_{[t-1]} \right) \\
  & \left(\bm{\overline{X}}^{\mathrm{c}}_{tt'i}, \bm{\widehat{X}}^{\mathrm{c}}_{tt'i}, \bm{\widehat{X}}^{\mathrm{b}}_{tt'i}, \bm{\widehat{Y}}_{tt'i}, \bm{\widehat{Z}}_{tt'i}\right) \in \Omega_{tt'i} \quad \forall \, t \in \mathcal{T}, \, t'=1,\dots,t, \, i \in \widehat{\mathcal{I}}_{t'},
\end{aligned}
\end{equation}
where $\bm{\widetilde{A}}^{\mathrm{c}}_{tt'} \in \mathbb{R}^{P_t \times \bar{N}_{t'}}$, $\bm{\widetilde{A}}^{\mathrm{b}}_{tt'} \in \mathbb{R}^{P_t \times \widehat{N}_{t'}}$, $\bm{\widetilde{D}}_{tt'} \in \mathbb{R}^{P_t \times |\bar{\mathcal{I}}_{t'+1}|}$, $\bm{\widetilde{H}}_{tt'} \in \mathbb{R}^{P_t \times |\widehat{\mathcal{I}}_{t'+1}|}$, and $\bm{\widetilde{B}}_{tt'} \in \mathbb{R}^{P_t \times K_{t'}}$.

\subsection{Reformulation}

Applying standard robust optimization techniques based on constraint-wise worst-case reformulation and linear programming (LP) duality \citep{Gorissen2015}, the semi-infinite program \eqref{eqn:MultistageApproxCompact} can be reformulated into the following problem:
\begin{equation}
\label{eqn:MultistageReformulation}
\begin{aligned}
\minimize_{\substack{\bm{\overline{X}}^{\mathrm{c}}, \, \bm{\widehat{X}}^{\mathrm{c}}, \, \bm{\widehat{X}}^{\mathrm{b}}, \\ \bm{\widehat{Y}}, \, \bm{\widehat{Z}}, \, \bm{\Phi}, \, \bm{\Psi}}} \quad & \bm{e}_1^{\top} \left(\bm{\overline{X}}^{\mathrm{c}}_{111} + \bm{\widehat{X}}^{\mathrm{c}}_{111} \right) \\
\st \quad & \sum_{t'=1}^t \sum_{i \in \widehat{\mathcal{I}}_{t'}} \bm{\psi}_{tt'i} \leq \bm{0} \quad \forall \, t \in \mathcal{T} \\
  & \bm{\Phi}_t \bm{w}_{tt'i} = - \bm{\tilde{b}}_{tt'i} \quad \forall \, t \in \mathcal{T}, \, t'=1,\dots,t, \, i \in \widetilde{\mathcal{I}}_{t'} \\
  & \bm{\psi}_{tt''i} + \bm{\Phi}_t \left[ \bm{\bar{w}}_{tt''i} v_{t''i} - \sum\limits_{t'=t''}^{t-1} \bm{U}^{\mathrm{c}}_{tt'} \bm{\overline{X}}^{\mathrm{c}}_{t't''i} \bm{\bar{v}}_{t''i} \right. \\
  & \quad \quad \quad \quad \quad \quad - \left. \sum\limits_{t'=t''}^{t-1} \left( \bm{U}^{\mathrm{c}}_{tt'} \bm{\widehat{X}}^{\mathrm{c}}_{t't''i} + \bm{U}^{\mathrm{b}}_{tt'} \bm{\widehat{X}}^{\mathrm{b}}_{t't''i} + \bm{\overline{U}}_{tt'} \bm{\widehat{Y}}_{t't''i} + \bm{\widehat{U}}_{tt'} \bm{\widehat{Z}}_{t't''i} \right) \bm{\hat{v}}_{t''i} \right] \\
  & \quad \quad \geq \sum\limits_{t'=t''}^t \left[ \bm{\widetilde{A}}^{\mathrm{c}}_{tt'} \bm{\overline{X}}^{\mathrm{c}}_{t't''i} \bm{\bar{v}}_{t''i} + \left(\bm{\widetilde{A}}^{\mathrm{c}}_{tt'} \bm{\widehat{X}}^{\mathrm{c}}_{t't''i} + \bm{\widetilde{A}}^{\mathrm{b}}_{tt'} \bm{\widehat{X}}^{\mathrm{b}}_{t't''i} + \bm{\widetilde{D}}_{tt'} \bm{\widehat{Y}}_{t't''i} + \bm{\widetilde{H}}_{tt'} \bm{\widehat{Z}}_{t't''i} \right) \bm{\hat{v}}_{t''i} \right] \\
  & \quad \quad \quad \quad \quad \quad \quad \quad \quad \quad \quad \quad \quad \quad \quad \quad \quad \quad \forall \, t \in \mathcal{T}, \, t''=1,\dots,t, \, i \in \widehat{\mathcal{I}}_{t''}, \, \bm{\tilde{v}}_{t''i} \in \widetilde{\mathcal{V}}_{t''i} \\
  & \bm{\psi}_{tt'i} \in \mathbb{R}^{P_t} \quad \forall \, t \in \mathcal{T}, \, t'=1,\dots,t, \, i \in \widehat{\mathcal{I}}_{t'} \\
  & \bm{\Phi}_t \in \mathbb{R}_+^{P_t \times Q_t} \quad \forall \, t \in \mathcal{T} \\
  & \left(\bm{\overline{X}}^{\mathrm{c}}_{tt'i}, \bm{\widehat{X}}^{\mathrm{c}}_{tt'i}, \bm{\widehat{X}}^{\mathrm{b}}_{tt'i}, \bm{\widehat{Y}}_{tt'i}, \bm{\widehat{Z}}_{tt'i}\right) \in \Omega_{tt'i} \quad \forall \, t \in \mathcal{T}, \, t'=1,\dots,t, \, i \in \widehat{\mathcal{I}}_{t'}
\end{aligned}
\end{equation}
where $\bm{w}_{tt'i}$, $\bm{\bar{w}}_{tt'i}$ $\bm{\tilde{b}}_{tt'i}$, and $\bm{\psi}_{ti}$ denote the columns of $\bm{W}_{tt'}$, $\bm{\overline{W}}_{tt'}$, $\bm{\widetilde{B}}_{tt'}$, and $\bm{\Psi}_t$, respectively, corresponding to uncertain parameter $i$. The detailed derivation of the reformulation can be found in the supplementary material.

Problem \eqref{eqn:MultistageReformulation} is generally a mixed-integer quadratically constrained program (MIQCP) as it involves bilinear terms. However, if the endogenous uncertainty set only depends on binary variables, the problem can be reformulated into a mixed-integer linear program (MILP) by expressing each integer variable as the difference between two binary variables and linearizing the bilinear terms, which will all be products of a continuous and a binary variable (for more details, see Feng et al. (2020)).

\section{Computational Case Studies}
\label{sec:CaseStudies}

In this section, we present the results from four computational case studies that demonstrate the versatility of the proposed modeling framework. All model instances were implemented in Julia v1.3.1 using the modeling environment JuMP v0.21.2 \citep{Lubin2015} and solved using Gurobi v8.1.1 on a Intel Core i7-8700 CPU at 3.20 GHz machine with 8 GB RAM.

\subsection{Pilot Plant and Plant Redesign}

We first consider an illustrative example that is based on Example \ref{exp:PilotPlant}. In this problem, we have a design of an industrial-scale plant whose capacity $q$ is uncertain and only becomes known after the plant is built. However, as described in Example \ref{exp:PilotPlant}, a pilot plant can be built such that its capacity $p$ helps provide a more accurate prediction of $q$. Once $p$ is observed, the marginal uncertainty set for $q$ is updated, based on which we can decide to build one or two plants using the current design or redesign the plant and build a plant using the new design. The solution has to ensure that a given demand $d$ can be met, and the objective is to optimize the worst case. The associated robust optimization problem can be formulated as follows:
\begin{equation}
\label{eqn:PilotPlant}
\begin{aligned}
\minimize_{x, y, z_1, z_2} \quad & \delta x + \max_{(p,q,\gamma) \in \Xi(x)} \gamma y(x p) + \rho\left[z_1(x p)+z_2(x p)\right] \\
\st \quad & x \in \{0,1\} \\
& \hspace{-8pt} \left.
\begin{array}{l}
  y(x p) \leq x \\[4pt]
  d \leq q \left[z_1(x p)+z_2(x p)\right] + d y(x p) \\[4pt]
  y(x p), \, z_1(x p), \, z_2(x p) \in \{0,1\}
\end{array}
\right\rbrace \quad \forall \, (p,q,\gamma) \in \Xi(x),
\end{aligned}
\end{equation}
where the binary variable $x$ equals 1 if the pilot plant is built and 0 otherwise, $y$ equals 1 if a plant with a new design and capacity $d$ is built, and $z_1$ and $z_2$ indicate how many (maximum two) plants with the current design are built. Nonnegative cost coefficients in the cost function are denoted by $\delta$, $\gamma$, and $\rho$, where $\gamma$ is an uncertain parameter as it depends on the pilot plant capacity. As indicated by the constraint $y(xp) \leq x$, a plant redesign is only possible if a pilot plant is built and hence more process knowledge is available. Finally, the total plant capacity has to be greater than or equal to demand $d$. The uncertainty set depends on $x$ and can be expressed as follows:
\begin{equation*}
\Xi(x) = \left\lbrace
(p,q,\gamma) \in \mathbb{R}_+^3:
\begin{array}{l}
  p^{\min} x \leq p \leq p^{\max} x \\ [4pt]
  q^{\min}(1-x) + (\alpha-\Delta)x + \beta p \leq q \leq q^{\max}(1-x) + (\alpha+\Delta)x + \beta p \\ [4pt]
  \gamma = (\rho + \tilde{\gamma} p^{\max}) x - \tilde{\gamma} p
\end{array}
\right\rbrace,
\end{equation*}
where $p$ and $\gamma$ can only take nonzero values if $x=1$. Similarly, $(\alpha-\Delta)+\beta p \leq q \leq (\alpha+\Delta)+\beta p$ if $x=1$, and $q^{\min} \leq q \leq q^{\max}$ if $x=0$. The parameters are given such that $q^{\min} = (\alpha-\Delta)+\beta p^{\min}$ and $q^{\max} = (\alpha+\Delta)+\beta p^{\max}$. We further assume that $q^{\min} < d \leq 2 q^{\min}$ and that there exists a $\hat{p} < p^{\max}$ such that $d = (\alpha-\Delta)+\beta \hat{p}$.

\subsubsection{Analytical Solution}
\label{sec:PilotPlantWCsolution}

Problem \eqref{eqn:PilotPlant} can be solved analytically, and the solution is illustrated in the decision diagram shown in Figure \ref{fig:PilotPlantDecisions}. If we choose to build a pilot plant ($x=1$), we can take recourse actions ($y$, $z_1$, and $z_2$) based on the realization of $p$. If $p \geq \hat{p}$, we know that $q \geq d$; hence, we only have to build one plant using the current design. If $p < \hat{p}$, building one plant with the current design may not be sufficient to meet demand $d$. In this case, we choose the less expensive option between building one plant with a new design that achieves a capacity of $d$ and building two plants with the current design. The costs for these two choices are $v = \delta + \rho + \tilde{\gamma}(p^{\max}-p)$ and $v = \delta + 2\rho$, respectively. Thus, the first option is less expensive than the second if $\tilde{\gamma}(p^{\max}-p) < \rho$. In contrast, if we do not build a pilot plant ($x=0$), there will be no recourse; in that case, in order to guarantee feasibility for every possible realization of $q$, we have to build two plants using the current design, which results in a cost of $v=2\rho$.

\begin{figure}[ht!]\centering
\includegraphics[width=6in]{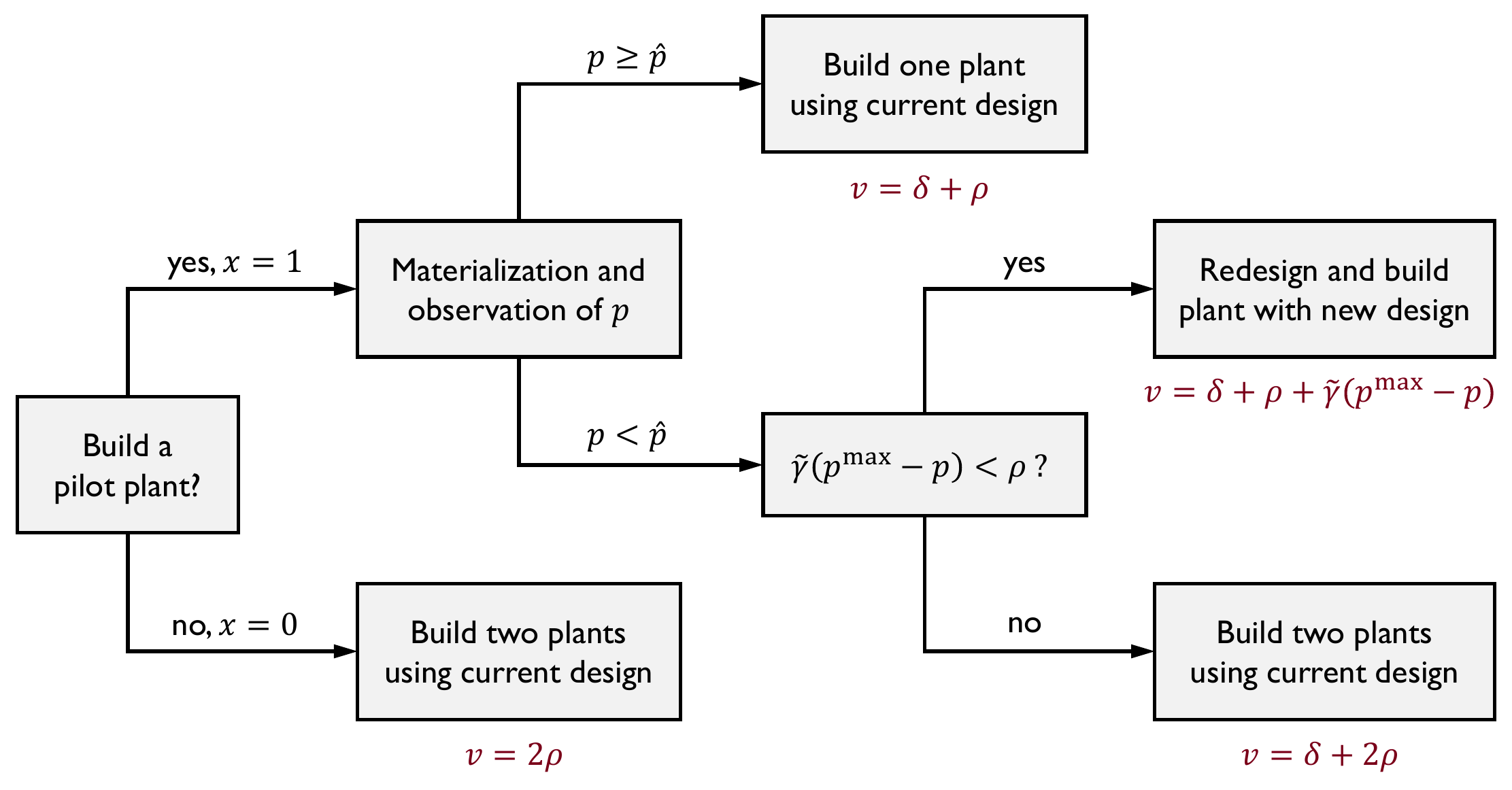}
\caption{Decision diagram for illustrative pilot plant example.}
\label{fig:PilotPlantDecisions}
\end{figure}

The problem further simplifies since the objective is to minimize the worst-case cost and the worst case is already known a priori, namely $p=p^{\min}$, $q = q^{\min}$, and $\gamma = \rho + \tilde{\gamma}(p^{\max}-p^{\min})$. Therefore, the optimal value of \eqref{eqn:PilotPlant} is simply
\begin{equation*}
v^* = \min \left\lbrace 2\rho, \, \delta+\rho+\tilde{\gamma}(p^{\max}-p^{\min}) \right\rbrace,
\end{equation*}
and the corresponding optimal solution is as follows: If $\rho < \delta+\tilde{\gamma}(p^{\max}-p^{\min})$, $x^* = y^* = 0$ and $z_1^* = z_2^* = 1$; if $\rho > \delta+\tilde{\gamma}(p^{\max}-p^{\min})$, $x^* = y^* = 1$ and $z_1^* = z_2^* = 0$. Both solutions are optimal if $\rho = \delta+\tilde{\gamma}(p^{\max}-p^{\min})$.

For our computational case study, we choose $\rho = 100$, $d = 59$, $p^{\min} = 5$, $p^{\max} = 10$, $q^{\min} = 35$, $q^{\max} = 75$, $\alpha = 10$, $\beta = 6$, and $\Delta = 5$. The heat map in Figure \ref{fig:Case1-WorstCaseObjective} shows $v^*$ as a function of $\delta$ and $\tilde{\gamma}$. The red line is given by $\delta = \rho - (p^{\max}-p^{\min})\tilde{\gamma} = 100 - 5 \tilde{\gamma}$, which indicates the boundary at which $v^*$ switches from $\delta+\rho+\tilde{\gamma}(p^{\max}-p^{\min})$ to $2 \rho$.

\begin{figure}[ht]\centering
\subfloat[Worst-case cost]{
  \label{fig:Case1-WorstCaseObjective}
	\fbox{\includegraphics[height=2.6in]{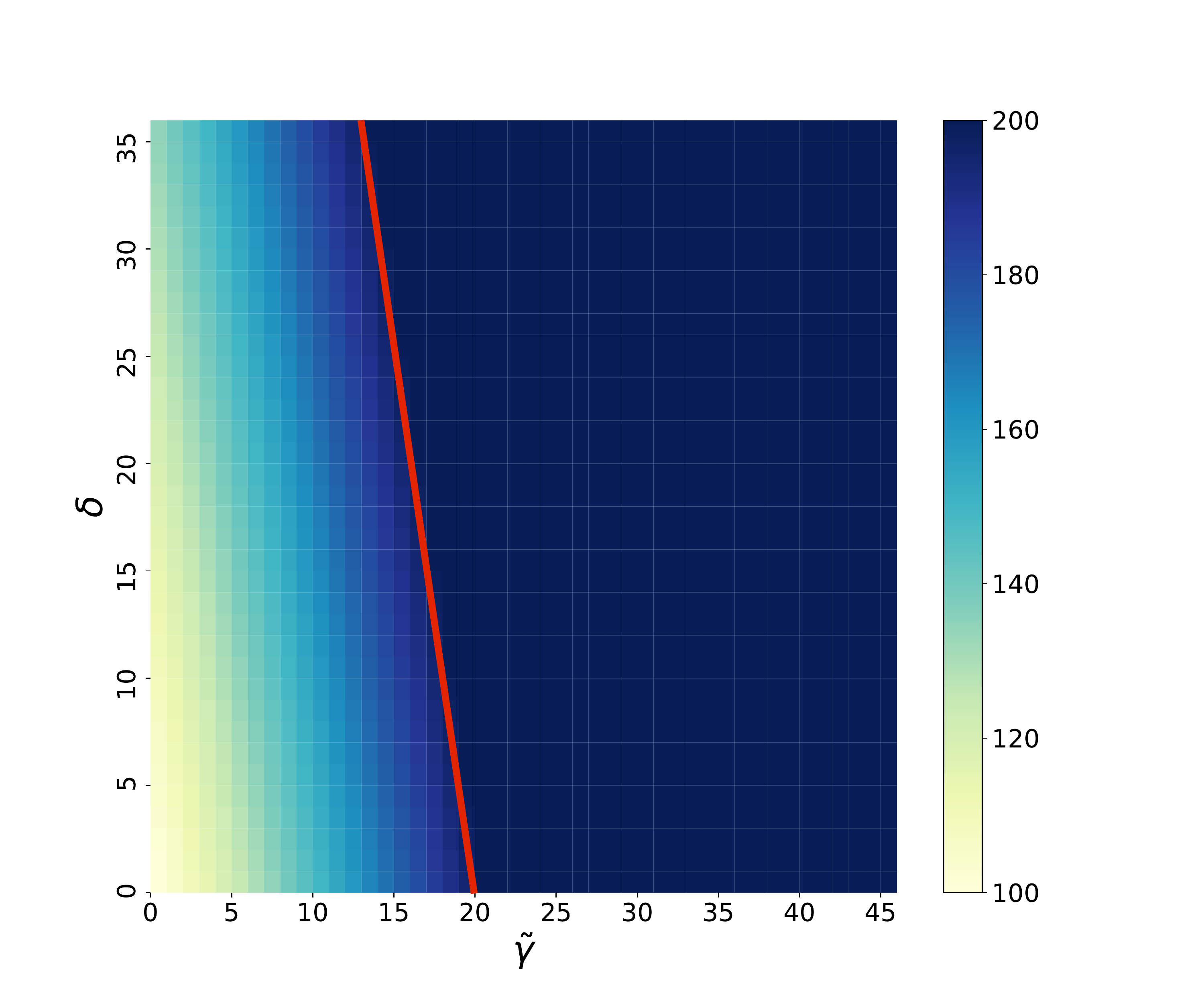}}} 
\subfloat[Approximate expected cost]{
  \label{fig:Case1-SAAObjective}
	\fbox{\includegraphics[height=2.6in]{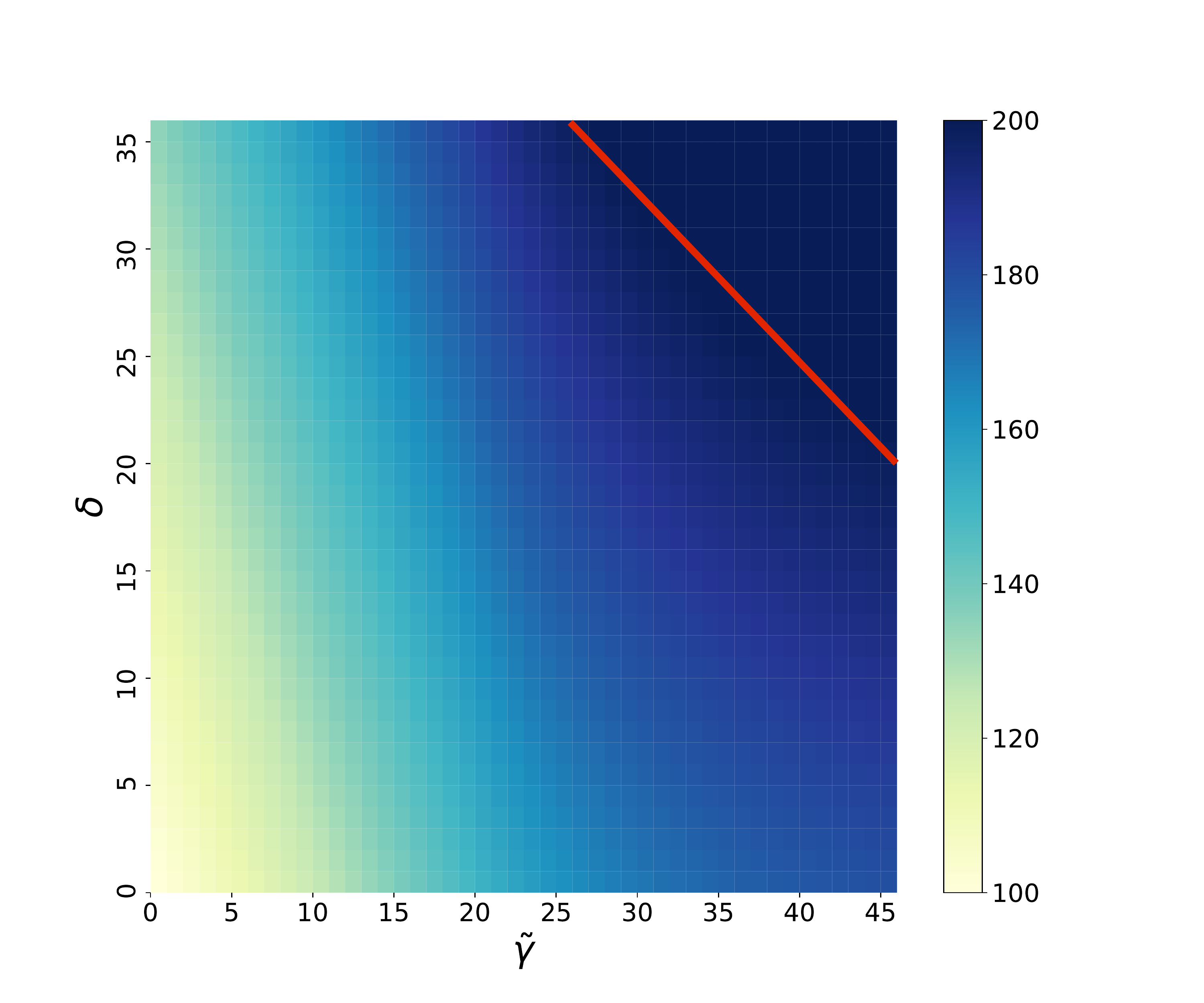}}}
\caption{The heat maps show the (a) worst-case cost and (b) scenario-based approximation of the expected cost as functions of $\delta$ and $\tilde{\gamma}$.}
\label{fig:Concept}
\end{figure}

\subsubsection{Non-Worst-Case Objective Function}

As mentioned above, in this example, recourse is only possible if $x=1$, which essentially means that this decision determines whether the problem is static or two-stage. However, note that recourse is not required in the optimal solution of \eqref{eqn:PilotPlant}, which is due to the fact that worst-case cost minimization is considered and that the worst case is known a priori. In the following, we consider an alternative, scenario-based objective function, which is often a better choice in practice. Here, we introduce a discrete scenario set $\mathcal{S}$, which contains possible realizations of $p$, each denoted by $\bar{p}_s$ with probability $\phi_s$, where $\sum_{s \in \mathcal{S}} \phi_s=1$. We then replace the objective function in problem \eqref{eqn:PilotPlant} with
\begin{equation}
\label{eqn:PilotPlantScenObj}
\delta x + \sum_{s \in \mathcal{S}} \phi_s \left( \bar{\gamma}_s y(x \bar{p}_s) + \rho\left[z_1(x \bar{p}_s)+z_2(x \bar{p}_s)\right] \right)
\end{equation}
with $\bar{\gamma}_s = \rho+\tilde{\gamma}(p^{\max}-\bar{p}_s)$. This objective function resembles a sample average approximation of the expected cost as commonly used in stochastic programming. We further reformulate the constraints to eliminate random recourse and remove $\gamma$ from the uncertainty set, arriving at the following formulation:
\begin{equation}
\label{eqn:PilotPlantRef}
\begin{aligned}
\minimize_{x, y, z_1, z_2, w_1, w_2} \quad & \delta x + \sum_{s \in \mathcal{S}} \phi_s \left( \bar{\gamma}_s y(x \bar{p}_s) + \rho\left[z_1(x \bar{p}_s)+z_2(x \bar{p}_s)\right] \right) \\
\st \quad \; & x \in \{0,1\} \\
& \hspace{-8pt} \left.
\begin{array}{l}
  y(x p) \leq x \\[4pt]
  w_1 (x p) \leq q \\[4pt]
  w_1 (x p) \leq q^{\max} z_1(x p) \\[4pt]
  w_2 (x p) \leq q \\[4pt]
  w_2 (x p) \leq q^{\max} z_2(x p) \\[4pt]
  d \leq w_1 (x p) + w_2 (x p) + d y(x p) \\[4pt]
  y(x p), \, z_1(x p), \, z_2(x p) \in \{0,1\} \\
  w_1 (x p), \, w_2 (x p) \in \mathbb{R}_+
\end{array}
\right\rbrace \quad \forall \, (p,q) \in \overline{\Xi}(x)
\end{aligned}
\end{equation}
with
\begin{equation*}\small
\overline{\Xi}(x) = \left\lbrace
(p,q) \in \mathbb{R}_+^2:
\begin{array}{l}
  p^{\min} x \leq p \leq p^{\max} x \\ [4pt]
  q^{\min}(1-x) + (\alpha-\Delta)x + \beta p \leq q \leq q^{\max}(1-x) + (\alpha+\Delta)x + \beta p 
\end{array}
\right\rbrace.
\end{equation*}
Note that we have introduced two new continuous recourse variables, $w_1$ and $w_2$, which represent the production amounts of plants built with the current design.

The optimal solution to this modified problem is less trivial as now recourse decisions, as depicted in Figure \ref{fig:PilotPlantDecisions}, play a crucial role. Still, we can solve it analytically. If we choose to build a pilot plant, the optimal recourse decisions are just as described in Section \ref{sec:PilotPlantWCsolution}. As a result, the optimal value is
\begin{equation*}
v^* = \min \left\lbrace 2\rho, \, \delta + \sum_{s \in \mathcal{S}_1} \phi_s \rho + \sum_{s \in \mathcal{S}_2} \phi_s [\rho + \tilde{\gamma}(p^{\max}-\bar{p}_s)] + \sum_{s \in \mathcal{S}_3} 2 \phi_s \rho \right\rbrace,
\end{equation*}
where $\mathcal{S}_1 := \{s \in \mathcal{S}: \bar{p}_s \geq \hat{p}\}$, $\mathcal{S}_2 := \{s \in \mathcal{S}: \bar{p}_s < \hat{p}, \, \tilde{\gamma}(p^{\max}-\bar{p}_s) < \rho\}$, and $\mathcal{S}_3 := \{s \in \mathcal{S}: \bar{p}_s < \hat{p}, \, \tilde{\gamma}(p^{\max}-\bar{p}_s) \geq \rho\}$ such that $\mathcal{S} = \mathcal{S}_1 \cup \mathcal{S}_2 \cup \mathcal{S}_3$.

We can also apply the proposed decision rule approach to solve this problem. To obtain the exact optimal solution, the marginal support of $p$ needs to be partitioned using appropriate breakpoints. Depending on the parameter values, we require at least one or two breakpoints. If $\tilde{\gamma}(p^{\max}-p^{\min}) \leq \rho$ or $\tilde{\gamma}(p^{\max}-\hat{p}) \geq \rho$, we only need one breakpoint at $\hat{p}$. If $\tilde{\gamma}(p^{\max}-p^{\min}) > \rho$ and $\tilde{\gamma}(p^{\max}-\hat{p}) < \rho$, we need two breakpoints, one at $p^{\max}-\rho/\tilde{\gamma}$ and one at $\hat{p}$. Note that the same decision rules are used across all scenarios in $\mathcal{S}$; as a result, compared to the worst-case formulation, no additional variables are added to the model such that a comparable computational complexity is achieved. We solve the problem using 100 scenarios sampled from a uniform distribution supported on the given uncertainty set, and obtain the heat map shown in Figure \ref{fig:Case1-SAAObjective}.

\subsection{Maintenance Planning with Inspections}

This example involves type-1, type-2a, and type-2b endogenous uncertainty. Here, we consider integrated production and maintenance planning with equipment degradation. We use the remaining useful life (RUL) as a measure of equipment health, which decreases with continued operation. To avoid equipment failure, the RUL has to remain greater than zero and can be restored through maintenance. Investing in an equipment upgrade, which would reduce the degradation rate, is also considered. However, degradation processes are inherently stochastic, which renders the change in RUL an uncertain parameter. Furthermore, RUL is an abstract quantity that is often difficult to estimate from operational process data such that elaborate inspections are required to obtain reliable equipment health information. Therefore, in this example, we assume that the RUL can only be observed through inspections, and formulate the resulting multistage robust optimization problem as follows:
\begin{equation}
\label{eqn:Maintenance}
\begin{aligned}
\minimize_{x, \bm{y}, \bm{p}, \bm{s}, \bm{r}, \bm{m}, \bm{f}, \bm{z}} \quad & \alpha x + \max_{(\bm{\xi},\bm{\bar{\xi}},\bm{\zeta}) \in \Theta_T(x,\bm{y},\bm{z})} \sum_{t \in \mathcal{T}} \left( \beta y_t + \gamma p_t + \delta z_t + \rho m_t - \varphi f_t \right) \\
\st \quad \;\; & x \in \{0,1\} \\
& \hspace{-8pt} \left.
\begin{array}{l}
  s_t = s_0 + \sum\limits_{t'=1}^t (p_{t'}-d_{t'}) \leq s^{\max} \\[4pt]
  p_t \leq C y_t \\[4pt]
  r_t = r_0 - \sum\limits_{t'=1}^t \xi_{t'} + \sum\limits_{t'=1}^t f_{t'} \leq r^{\max} \\[4pt]
  f_t \leq r^{\max} m_t \\[4pt]
  y_t + m_t \leq 1 \\[4pt]
  m_t \leq z_{t-1} \\[4pt]
  p_t, \, s_t, \, r_t, \, f_t \in \mathbb{R}_+, \quad y_t, \, z_t, \, m_t \in \{0,1\}
\end{array}
\right\rbrace
\begin{array}{l}
  \forall \, t \in \mathcal{T}, \\[4pt]
  (\bm{\xi}_{[t]},\bm{\bar{\xi}}_{[t]},\bm{\zeta}_{[t]}) \in \Theta_t(x,\bm{y}_{[t]},\bm{z}_{[t]}),
\end{array}  
\end{aligned}
\end{equation}
where $x$ equals 1 if an equipment upgrade is performed, $y_t$ equals 1 if the plant operates in time period $t$, $p_t$ is the amount produced in time period $t$, $s_t$ is the inventory level at time $t$ (end of time period $t$), $r_t$ is the RUL at time $t$, $m_t$ equals 1 if maintenance is performed in time period $t$, $f_t$ is the amount of RUL restored in time period $t$, and $z_t$ equals 1 if an inspection is conducted at time $t$. The recourse variables are functions of uncertain parameters; however, for the sake of brevity, we do not explicitly indicate this in \eqref{eqn:Maintenance}. With given positive cost coefficients $\alpha$, $\beta$, $\gamma$, $\delta$, $\rho$, and $\varphi$, the objective is to minimize the worst-case cost over the planning horizon given by the set of time periods $\mathcal{T}=\{1,\dots,T\}$. Note that a negative cost is assigned to $f_t$ in order to encourage maintenance toward the end of the equipment's residual life and to retain a reasonable RUL value at the end of the planning horizon. Given an initial inventory $s_0$ and demand $d_t$ for every time period $t$, the second constraint in \eqref{eqn:Maintenance} expresses material balance and sets an upper bound $s^{\max}$ on the inventory. The next constraint states that the production amount is bounded by the plant capacity $C$. Assuming the initial RUL $r_0$ is known, $r_t$ is computed and bounded in the next constraint. The RUL reduction caused by operation in time period $t$ is denoted by $\xi_t$, which is uncertain. If maintenance is performed, the RUL can be restored to its maximum $r^{\max}$, as stated in the next inequality. Finally, the last two constraints state that maintenance can only be performed if the plant is shut down and needs to be immediately preceded by an inspection. We further assume that $z_0 = m_1 = f_1 = 0$.

The decision-dependent uncertainty set, which includes auxiliary uncertain parameters, is
\begin{equation*}
\Theta_t(x,\bm{y}_{[t]},\bm{z}_{[t]}) = \left\lbrace
(\bm{\xi}_{[t]},\bm{\bar{\xi}}_{[t]},\bm{\zeta}_{[t]}) \in \mathbb{R}^{3t}_+:
\begin{array}{l}
  \bar{\xi}_{t'} = \sum\limits_{t''=1}^{t'} \xi_{t''} \quad \forall \, t'=1,\dots,t \\ [8pt]
  \xi^{\min} y_{t'} \leq \xi_{t'} \leq \xi^{\max} y_{t'} \quad \forall \, t'=1,\dots,t \\ [4pt]
  \xi_{t'} \leq \xi^{\max} - \xi^{\Delta} x \quad \forall \, t'=1,\dots,t \\ [4pt]
  \bar{\xi}_{t'} - t' \xi^{\max} (1-z_{t'}) \leq \zeta_{t'} \leq \bar{\xi}_{t'} \quad \forall \, t'=1,\dots,t \\[4pt]
  \zeta_{t'} \leq \bar{\xi}^{\max}_{t'} z_{t'} \quad \forall \, t'=1,\dots,t
\end{array}
\right\rbrace,
\end{equation*}
where the auxiliary uncertain parameter $\bar{\xi}_t$ is the cumulative RUL reduction up to time $t$. The reason for introducing $\bar{\xi}_t$ is that an inspection at time $t$ reveals $\bar{\xi}_t$ but not each individual $\xi_{t'}$ that has contributed to the cumulative RUL reduction since the last inspection. The RUL reduction in time period $t$, $\xi_t$, is zero if the plant does not operate, i.e. $y_t=0$, and bounded by $\xi^{\min}$ and $\xi^{\max}$ otherwise. This upper bound is further reduced by $\xi^{\Delta}$ if an equipment upgrade is performed, i.e. $x=1$. We have another auxiliary variable $\zeta_t$ for every $t \in \mathcal{T}$, which takes the value of $\bar{\xi}_t$ if it is observed through an inspection at time $t$, i.e. $z_t = 1$, and is zero otherwise.

Note that $\mathcal{T}$ denotes the set of time periods, not the set of stages. In fact, strictly speaking, the number of stages depends on the solution and is directly related to the number of inspections. The maximum possible number of stages is $T$, which is only achieved if inspection is performed at every $t \in \mathcal{T}$. In that case, after eliminating the state variables $s_t$ and $r_t$, the first-stage decisions are $x$, $y_1$, $p_1$, and $z_1$. Then, $\bar{\xi}_1$ is observed and the second-stage decisions are $y_2$, $p_2$, $m_2$, $f_2$, and $z_2$; this continues similarly up to stage $T$. The number of stages is implicitly encoded in the recourse variables that are functions of only observed uncertain parameters. Specifically, for all $t \in \mathcal{T}$, $y_t = y_t(\bm{\zeta}_{[t-1]})$, $p_t = p_t(\bm{\zeta}_{[t-1]})$, $m_t = m_t(\bm{\zeta}_{[t-1]})$, $f_t = f_t(\bm{\zeta}_{[t-1]})$, and $z_t = z_t(\bm{\zeta}_{[t-1]})$.

We conduct a case study with the following data: $T = 7$, $\beta = \$\,2 \times 10^3$, $\gamma = \$\,400$, $\delta = \$\,10^3$, $\rho = \$\,2.5 \times 10^5$, $\varphi = \$\,2 \times 10^3$, $s_0 = 10$, $s^{\max} = 50$, $C = 90$, $r_0 = 40$, $r^{\max} = 125$, $d_t = 50$ for all $t \in \mathcal{T}$, $\xi^{\min} = 4$, $\xi^{\max} = 25$, and $\xi^{\Delta} = 13$. Two cases are considered: Case 1, in which the equipment upgrade cost $\alpha$ is $\$\,5 \times 10^4$, and Case 2 with $\alpha$ being $\$\, 10^4$. For each case, we solve one instance using the above ARO model with one breakpoint placed at 8 for each $\xi_t$, and another instance that disallows recourse, resulting in a static problem (SRO). The optimal values and solution times for all four instances are shown in Table \ref{tab:Example2}. In both Cases 1 and 2, one can observe a clear benefit from recourse, which is reflected in an improvement in the optimal value of about 3\,\%. However, the computational cost for solving the ARO model is significantly higher due to the larger model size. Specifically, the ARO model has 218 binary variables, 320,150 continuous variables, and 952,374 constraints after Gurobi's preprocessing.

\begin{table}[ht!]\centering
\setlength\tabcolsep{6pt}
	\caption{Computational results for instances of the maintenance planning problem.}
	\begin{tabular}{ccccc}
	\toprule
    & \makecell{\textbf{Cost of} \\ \textbf{upgrade (k\$)}} & \textbf{Model} & \makecell{\textbf{Optimal} \\ \textbf{value (k\$)}}  & \makecell{\textbf{Solution} \\ \textbf{time (s)}}  \\
    \midrule
    \multirow{2}{*}{\textbf{Case 1}} & \multirow{2}{*}{50} & SRO & 219 & 425 \\
    & & ARO  & 212 & 13,621 \\
    \midrule
    \multirow{2}{*}{\textbf{Case 2}} & \multirow{2}{*}{10} & SRO & 213 & 382 \\
    & & ARO   & 206 & 7,837 \\
 	\bottomrule
	\end{tabular}
	\label{tab:Example2}
\end{table}

The results in Table \ref{tab:Example2} also indicate that an equipment upgrade can be beneficial if the cost is sufficiently low as the solution suggests investing in an equipment upgrade in Case 2 but not in Case 1. Figure \ref{fig:Example2} further shows how the equipment upgrade affects the optimal schedule. As depicted in Figure \ref{fig:NoUpgrade}, without equipment upgrade, inspection is performed at time 1. Then, depending on the realization of $\bar{\xi}_1$, the top or bottom schedule is chosen. If $\bar{\xi}_1 \geq 8$, maintenance is performed in time period 2, immediately after the inspection; otherwise, maintenance is postponed to time period 3. In contrast, the solution for Case 2, as shown in Figure \ref{fig:Upgrade}, suggests inspecting the equipment at time 2. This later time point is made possible by the equipment upgrade that reduces the maximum degradation rate. Then, if $\bar{\xi}_2 \geq 16$, maintenance is performed in time period 3; otherwise, maintenance is scheduled for time period 6.

\begin{figure}[ht]\centering
\subfloat[ARO solution for Case 1]{
  \label{fig:NoUpgrade}
	\fbox{\includegraphics[height=1.9in]{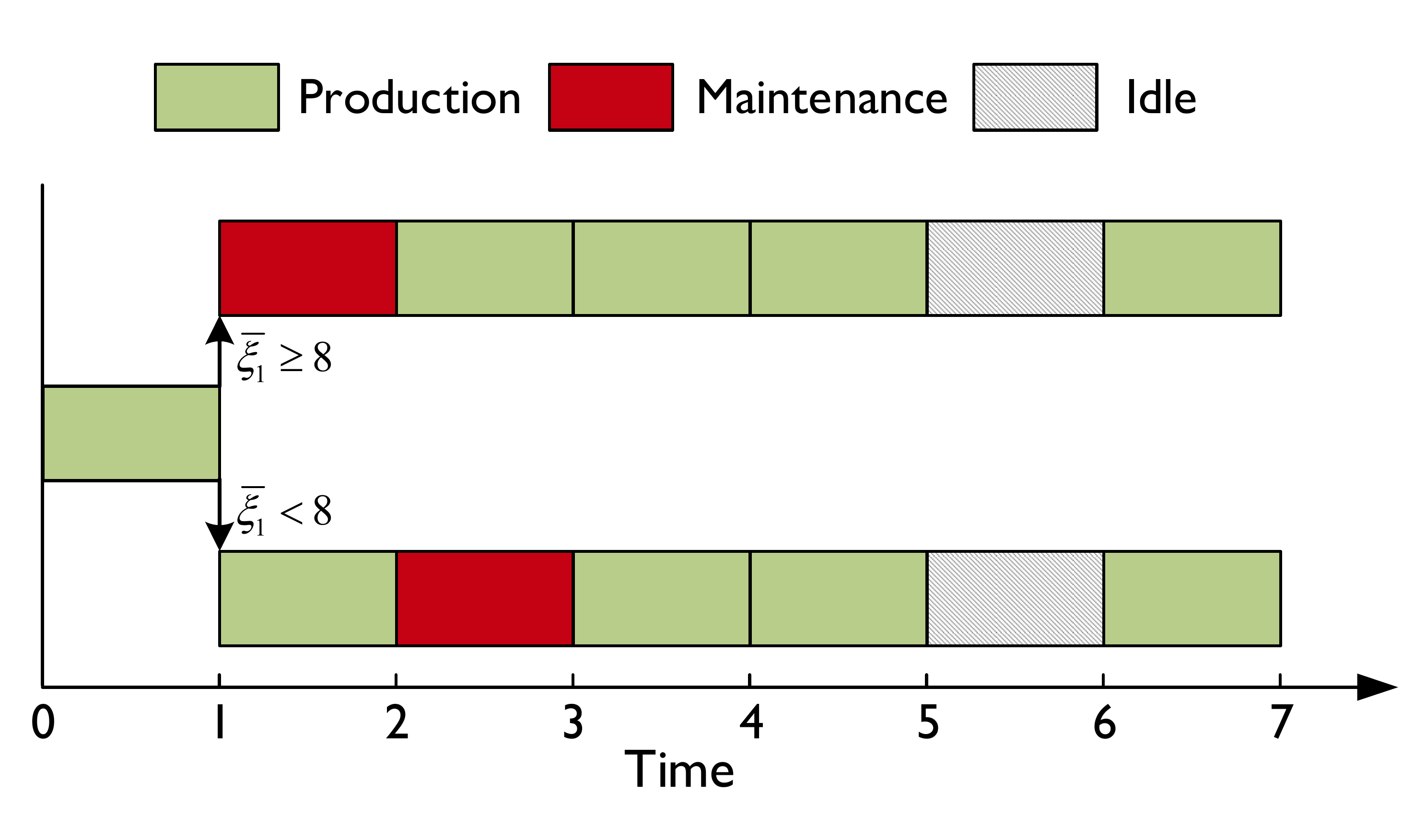}}} 
\subfloat[ARO solution for Case 2]{
  \label{fig:Upgrade}
	\fbox{\includegraphics[height=1.9in]{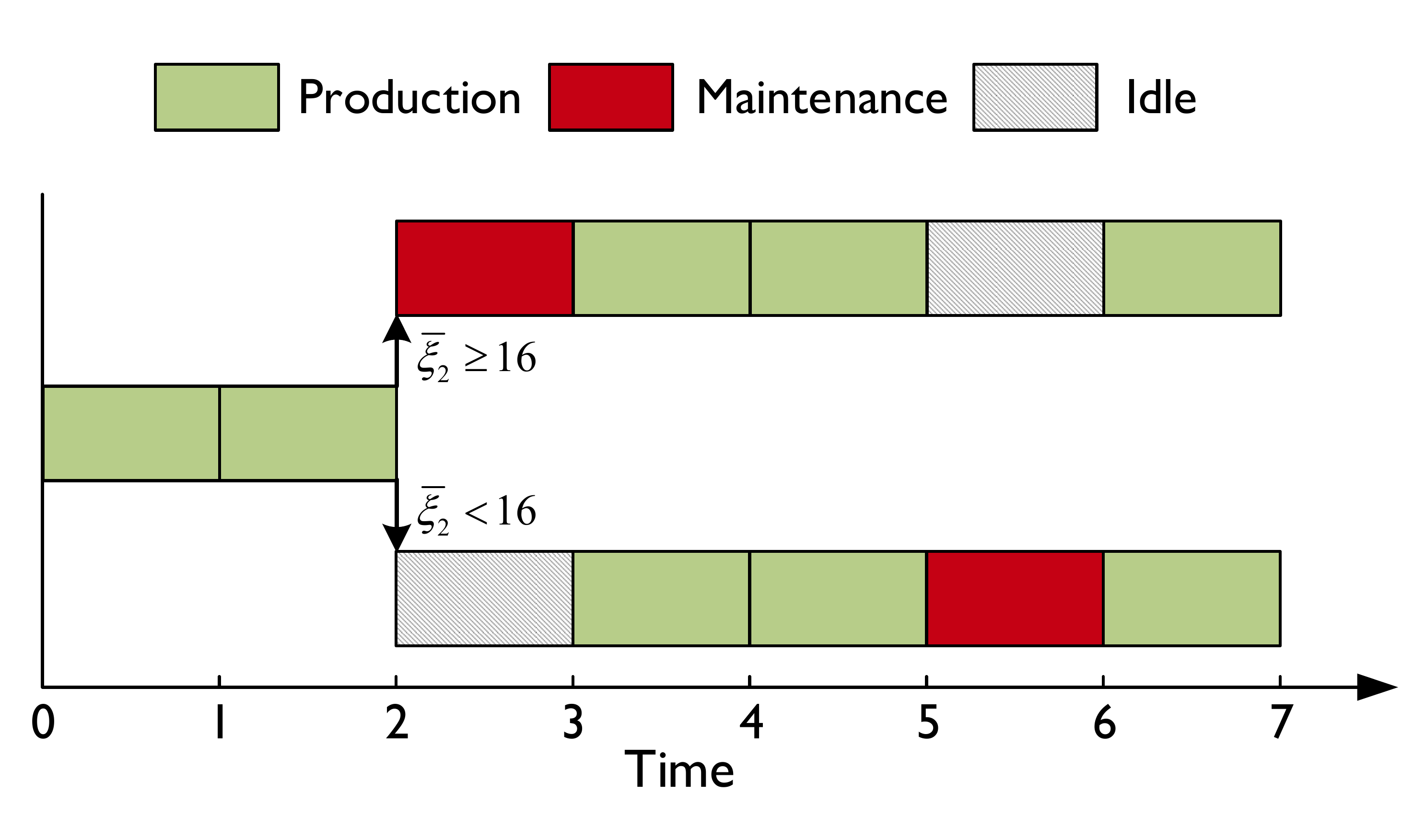}}}
\caption{The Gantt charts show the schedules obtained from solving the ARO model for (a) Case 1 with high and (b) Case 2 with low cost of equipment upgrade.}
\label{fig:Example2}
\end{figure}

\subsection{Optimizing Revision Points in Capacity Planning}

The previous example considers a problem in which the number of stages is affected by decisions due to the decision-dependent observation of uncertain parameters. There are other situations in which uncertain parameters are observed over time, independent of decisions, yet we still would like to restrict or optimize the number of stages. One example is long-term capacity planning, which involves large investments that require commitments and lead time for establishing contracts and infrastructure. Hence, it is often infeasible to change decisions frequently despite new information being constantly revealed over time. Instead, one has to restrict oneself to a few revision points at which the planning decisions are revised and updated. In their recent work, \citet{Basciftci2019} have formally introduced an \textit{adaptive two-stage stochastic programming} approach in which for each decision policy, one revision point is considered and chosen as part of the optimization problem. This problem can be viewed as a stochastic program with type-2b endogenous uncertainty where planning decisions for time periods before the revision point have to be made here and now while decisions after the revision point can be adjusted based on the uncertainty revealed up to the revision point. As a result, our framework can be directly applied to solve a robust optimization variant of the problem and extensions that consider multiple revision points. 

Consider a simplified power generation capacity expansion problem with $T$ time periods and a set of types of generation resources $\mathcal{I}$. At the beginning of each time period $t$ (i.e. at time $t-1$), the uncertain electricity demand for the same time period, $d_t$, is revealed and the amount of electricity generated from each generation type $i$ in time period $t$, $y_{it}$, is determined. Capacity expansion decisions, however, can only be made or updated at revision points. Let $x_{it}$ denote the amount of generation capacity of type $i$ that is added in time period $t$ and becomes available at time $t$. Assume that for each generation type $i$, there are two revision points, one at time $t'_i-1$, immediately after $d_{t'_i}$ is revealed, and the other one at time $t''_i-1$ with $t'_i<t''_i$. Then, all $x_{it}$ with $t < t'_i$ have to be chosen before any demand for the planning horizon is known, all $x_{it}$ with $t'_i \leq t < t''_i$ are chosen at time $t'_i-1$ with known $\bm{d}_{[t'_i]}$, and all $x_{it}$ with $t \geq t''_i$ are chosen at time $t''_i-1$ with known $\bm{d}_{[t''_i]}$. Given a maximum number of revision points for each generation type $i$, denoted by $n_i$, we aim to place the revision points such that a given cost function, in this case a sample average approximation of the expected cost, is minimized.

We introduce binary variables $\hat{z}_{it}$, which equals 1 if capacity expansion decisions for generation type $i$ are revised at time $t-1$, and $z_{itt'}$, which equals 1 if $d_{t'}$ is known when the decision $x_{it}$ is made. The problem can then be formulated as follows:
\begin{subequations}
\label{eqn:CapacityPlanning}
\begin{align}
\minimize_{\bm{z}, \bm{\hat{z}}, \bm{x}, \bm{y}} \quad & \sum_{i \in \mathcal{I}} \alpha_i x_{i0} + \sum_{s \in \mathcal{S}} \phi_s \sum_{t \in \mathcal{T}} \frac{1}{(1+r)^t} \left[ \sum_{i \in \mathcal{I}} \left( \alpha_i \bar{x}_{its} + \beta_i \bar{y}_{its} \right) \right] \label{eqn:CPobj} \\
\st \quad \;\; & \sum_{t \in \mathcal{T}} \hat{z}_{it} \leq n_i \quad \forall \, i \in \mathcal{I} \label{eqn:CPrevision} \\
  & \frac{1}{n_i} \sum_{t''=t'}^t \hat{z}_{it''} \leq z_{itt'} \leq \sum_{t''=t'}^t \hat{z}_{it''} \quad \forall \, i \in \mathcal{I}, \, t \in \mathcal{T}, \, t'=1,\dots,t \label{eqn:CPobserved} \\
  & \hat{z}_{it} \in \{0,1\} \quad \forall \, i \in \mathcal{I}, \, t \in \mathcal{T} \\
  & z_{itt'} \in \{0,1\} \quad \forall \, i \in \mathcal{I}, \, t \in \mathcal{T}, \, t'=1,\dots,t \\
  & x_{it} \leq C_{it} \quad \forall \, i \in \mathcal{I}, \, t \in \mathcal{T} \label{eqn:CPcapbound} \\
  & y_{it} \leq \sum\limits_{t'=1}^t x_{i,t'-1} \quad \forall \, i \in \mathcal{I}, \, t \in \mathcal{T}, \, (\bm{d}_{[t]}, \bm{\zeta}_{[t]}) \in \Theta_t(\bm{z}_{[t]}) \label{eqn:CPcapacity} \\
  & \sum\limits_{i \in \mathcal{I}} y_{it} \geq d_t \quad \forall \,  t \in \mathcal{T}, \, (\bm{d}_{[t]}, \bm{\zeta}_{[t]}) \in \Theta_t(\bm{z}_{[t]}) \label{eqn:CPbalance} \\
  & x_{i0} \in \mathbb{R}_+ \quad \forall \, i \in \mathcal{I} \\
  & y_{it}, \, x_{it} \in \mathbb{R}_+ \quad \forall \, i \in \mathcal{I}, \, t \in \mathcal{T}, \, (\bm{d}_{[t]}, \bm{\zeta}_{[t]}) \in \Theta_t(\bm{z}_{[t]}),
\end{align}
\end{subequations}
where $\mathcal{T}=\{1,\dots,T\}$, $\mathcal{S}$ denotes the set of scenarios, $\alpha_i$ and $\beta_i$ are cost coefficients, and $r$ is a discount factor. Each scenario $s \in \mathcal{S}$ is characterized by a specific electricity demand profile $\bm{\bar{d}}_{[T],s}$ and probability $\phi_s$. In the objective function \eqref{eqn:CPobj}, $\bar{x}_{its} = x_{it}(\bm{z}_{it} \circ \bm{\bar{d}}_{[t],s})$ and $\bar{y}_{its} = y_{it}(\bm{\bar{d}}_{[t],s})$ represent the recourse decisions for specific scenarios. Constraints \eqref{eqn:CPrevision} restrict the number of revision points. Inequalities \eqref{eqn:CPobserved} force $z_{itt'}$ to be 1 if there is at least one revision point between times $t'-1$ and $t-1$ and 0 otherwise. The rationale is that if there is one or multiple revision points between $t'-1$ and $t-1$, then $x_{it}$ will be determined at the latest revision point within that time frame at which point $d_{t'}$ will be known. Constraints \eqref{eqn:CPcapbound} set an upper bound, $C_{it}$, on each $x_{it}$. Capacity constraints are given in \eqref{eqn:CPcapacity}, and constraints \eqref{eqn:CPbalance} state that the total amount of electricity produced has to meet or exceed the demand. The decision-dependent uncertainty set is
\begin{equation*}\small
\Theta_t(\bm{z}_{[t]}) = \left\lbrace
(\bm{d}_{[t]}, \bm{\zeta}_{[t]}) \in \mathbb{R}_+^{K_t}:
\begin{array}{l}
  d_{1}^{\min} \leq d_{1} \leq d_{1}^{\max}  \\ [4pt]
  \delta_{t'}^{\min} d_{t'-1} \leq d_{t'} \leq \delta_{t'}^{\max} d_{t'-1} \quad \forall \, t'=2,\dots,t \\ 
  \sum\limits_{t''=1}^{t'} d_{t''} \leq \sum\limits_{t''=1}^{t'} \bar{\delta}^{t''-1} d_1 \quad \forall \, t'=1,\dots,t \\ [8pt]
  d_{t''} - d^{\max}_{t''} (1-z_{it't''}) \leq \zeta_{it't''} \leq d_{t''} \quad \forall \, i \in \mathcal{I}, \, t'=1,\dots,t, \, t''=1,\dots,t' \\[4pt]
  d^{\min}_{t''} z_{it't''} \leq \zeta_{it't''} \leq d^{\max}_{t''} z_{it't''} \quad \forall \, i \in \mathcal{I}, \, t'=1,\dots,t, \, t''=1,\dots,t'
\end{array}
\right\rbrace,
\end{equation*}
which encodes the correlation between the electricity demands in different time periods as the range in which $d_{t'}$ can vary depends on $d_{t'-1}$ and the budget of uncertainty imposed on $\sum_{t''=1}^{t'} d_{t''}$ is $\sum_{t''=1}^{t'} \bar{\delta}^{t''-1} d_1$, which is a function of $d_1$ with a given parameter $\bar{\delta}$. The auxiliary uncertain parameters $\bm{\zeta}_{[t]}$ depend on the choice of revision points and hence on $\bm{z}_{[t]}$. Note that the vector $\bm{z}_{[t]}$ contains all $z_{it't''}$ with $i \in \mathcal{I}$, $t'=1,\dots,t$, and $t''=1,\dots,t'$. Also, $K_t=t+|\mathcal{I}|t(t+1)/2$. The recourse variables are adjustable given the following dependence on uncertain parameters: $x_{it} = x_{it}(\bm{\zeta}_{[t]})$ and $y_{it} = y_{it}(\bm{d}_{[t]})$.

Notice that $\bm{z}$ and $\bm{\hat{z}}$, which determine the revision points, are all first-stage variables. As a result, the decision-dependent nonanticipativity for the capacity expansion decisions $\bm{x}$ can alternatively be modeled by forcing the decision rule coefficients associated with unobserved uncertain parameters to zero (see Section \ref{sec:Nonanticipativity}). The decision rules with respect to the lifted uncertain parameters are
\begin{equation*}
  x_{it} = \sum_{t'=0}^t \left( \bm{\bar{x}}_{itt'}^{\top} \bm{\bar{d}}_{t'} + \bm{\hat{x}}_{itt'}^{\top} \bm{\hat{d}}_{t'} \right) \quad \forall \, i \in \mathcal{I}, \, t \in \mathcal{T}
\end{equation*}
with $\bar{d}_0=1$ and $\hat{d}_0=0$. Decision-dependent nonanticipativity is then enforced by the following constraints:
\begin{equation*}
\begin{aligned}
  -M z_{itt'} \bm{e} \leq \bm{\bar{x}}_{itt'} \leq M z_{itt'} \bm{e} \quad \forall \, i \in \mathcal{I}, \, t \in \mathcal{T}, \, t'=1,\dots,t \\
  -M z_{itt'} \bm{e} \leq \bm{\hat{x}}_{itt'} \leq M z_{itt'} \bm{e} \quad \forall \, i \in \mathcal{I}, \, t \in \mathcal{T}, \, t'=1,\dots,t
\end{aligned}
\end{equation*}
with $M$ being a sufficiently large parameter. This alternative approach allows us to work with a fixed uncertainty set:
\begin{equation*}
\Theta_t = \left\lbrace
\bm{d}_{[t]} \in \mathbb{R}_+^t:
\begin{array}{l}
  d_{1}^{\min} \leq d_{1} \leq d_{1}^{\max}  \\ [4pt]
  \delta_{t'}^{\min} d_{t'-1} \leq d_{t'} \leq \delta_{t'}^{\max} d_{t'-1} \quad \forall \, t'=2,\dots,t \\ 
  \sum\limits_{t''=1}^{t'} d_{t''} \leq \sum\limits_{t''=1}^{t'} \bar{\delta}^{t''-1} d_1 \quad \forall \, t'=1,\dots,t
\end{array}
\right\rbrace.
\end{equation*}
Moreover, it does not require auxiliary uncertain parameters, which significantly reduces the computational complexity.


In the following case study, we consider a planning horizon of seven time periods and four generation types, namely coal, natural gas-gas turbine (NG-GT), solar PV, and wind. Cost data are adapted from \citet{Min2018}, which, along with all other required data, can be found in the supplementary material. Furthermore, the objective function is constructed using 100 scenarios sampled from a uniform distribution supported on the given uncertainty set, and uncertainty lifting is performed with two equidistantly placed breakpoints for each uncertain parameter.

We solve the capacity planning problem for different maximum allowed number of revision points $\bar{n}$, which is applied to all generation types, i.e. $n_i = \bar{n}$ for all $i \in \mathcal{I}$. The results are shown in Table \ref{tab:Case3-SAA}. As expected, the optimal value decreases with increasing number of revision points. The improvement is most significant from the static case (no revision) to the case with one revision point. Diminishing returns are observed at higher number of revision points. The results show that the locations of the revision points for different generation types are generally not the same. For larger $\bar{n}$, there also seems to be no need to use the maximum number of revision points for all generation types. Finally, note that the solution times are moderate, indicating the potential to solve more complex problems of this kind using the proposed approach.

\begin{table}[ht!]\centering
\setlength\tabcolsep{2.5pt}	
	\caption{Computational results for the capacity planning problem with different allowed numbers of revision points. The relative improvement is computed with respect to the optimal value of the case with no revision.}
	\begin{tabular}{cccccccc}
	\toprule
	\multirow{2}{*}{\makecell{\textbf{Allowed \# of} \\ \textbf{revision points}}} & \multirow{2}{*}{\makecell{\textbf{Optimal} \\ \textbf{value (k\$)}}} & \multirow{2}{*}{\makecell{\textbf{Solution} \\ \textbf{time (s)}}} & \multirow{2}{*}{\makecell{\textbf{Relative} \\ \textbf{improvement}}} & \multicolumn{4}{c}{\textbf{Location of revision points}} \\ 
	& & & & Coal  & NG-GT & Solar PV & Wind \\ \midrule
    0 & 46,183 	& 9 	&   &  &  &  &   \\
    1 & 40,852 	&427 	&11.54\,\%	&   3	&   5	&   1	&   2 \\
    2 & 39,526 	& 1,679 	& 14.42\,\%	&   3, 6	&   2, 5	&   1, 6	&   1, 3 \\
    3 & 38,955 	& 764 	&15.65\,\%	&  2, 3, 6	&   1, 2, 5	&  1, 5, 6	&  1, 3, 5 \\
    4 & 38,612 	& 874 	& 16.39\,\%	&   2, 3, 5, 6	&   1, 2, 3, 5	&   1, 5, 6	&   1, 3, 5 \\
    5 & 38,423 	&592 	&16.80\,\%	&  1, 2, 3, 5, 6	&  1, 2, 3, 4, 5	&  5, 6	&  1, 2, 3, 5 \\
    6 & 38,423 	&462 	&16.80\,\%	&   1, 2, 3, 4, 5, 6	&  1, 2, 3, 4, 5	&  5, 6 & 1, 2, 3, 5 \\
	\bottomrule
	\end{tabular}
	\label{tab:Case3-SAA}
\end{table}


\subsection{Production Scheduling with Active Parameter Estimation}

When operating a manufacturing process, we often do not have accurate information across the whole range of feasible operation due to the lack of process data. A model that we use for process optimization may only be accurate around one or a few nominal points, especially if it is a relatively new process. For example, a production unit may be able to operate in different operating modes, but process parameters such as efficiency and processing time are only accurately known for the one operating mode that is used most of the time. The models of other, potentially better-performing operating modes can only be improved by actually operating in those modes and collecting more data; however, this exploration action comes with the risk of operating under conditions with inferior performance or damaging impact. In the following example, we consider process scheduling with active model parameter estimation, which integrates scheduling optimization and design of experiments in order to find an optimal trade-off between exploration and exploitation that maximizes the overall plant performance over time. The problem is formulated as follows:
\begin{equation}
\label{eqn:SchedulingEstimation}
\begin{aligned}
\minimize_{\bm{x}, \bm{s}} \quad & \max_{\bm{p} \in \mathcal{P}_T(\bm{x})} \sum_{t \in \mathcal{T}} \sum_{i \in \mathcal{I}} \sum_{j \in \mathcal{J}_i} \bm{c}^{\top} \bm{u}_{ij} \, x_{tij}(\bm{p}_{[t]}) + \delta s_t(\bm{p}_{[t]}) \\
\st \quad & s_t(\bm{p}_{[t]}) = s_0 + \sum_{t'=1}^t \left( \sum_{i \in \mathcal{I}} \sum_{j \in \mathcal{J}_i} p_{t'ij} - d_{t'} \right) \leq s^{\max} \quad \forall \, t \in \mathcal{T}, \, \bm{p}_{[t]} \in \mathcal{P}_t(\bm{x}_{[t]}) \\
& \sum_{i \in \mathcal{I}} \sum_{j \in \mathcal{J}_i} x_{tij}(\bm{p}_{[t]}) \leq 1  \quad \forall \, t \in \mathcal{T}, \, \bm{p}_{[t]} \in \mathcal{P}_t(\bm{x}_{[t]}) \\
&  s_t(\bm{p}_{[t]}) \in \mathbb{R}_+ \quad \forall \, t \in \mathcal{T}, \, \bm{p}_{[t]} \in \mathcal{P}_t(\bm{x}_{[t]}) \\
& x_{tij}(\bm{p}_{[t]}) \in \{0,1\} \quad \forall \, t \in \mathcal{T}, \, i \in \mathcal{I}, \, j \in \mathcal{J}_i, \, \bm{p}_{[t]} \in \mathcal{P}_t(\bm{x}_{[t]}),
\end{aligned}
\end{equation}
where we assume that the plant operating point is chosen among a discrete set of operating points, $\mathcal{J}_i$, which depends on operating mode $i$. Each operating point $j$ of mode $i$ is defined by a vector of inputs, $\bm{u}_{ij}$. The binary variable $x_{tij}$ equals 1 if the plant operates at point $j$ of mode $i$ in time period $t$. The amount produced, $p_{tij}$, is treated as an uncertain parameter since the model parameters for the corresponding mode $i$ may be uncertain. Given cost coefficients $\bm{c}$ and $\delta$, the objective is to minimize the worst-case cost over the entire scheduling horizon $\mathcal{T} = \{1,\dots,T\}$. With an initial inventory $s_0$, the inventory level at time $t$, $s_t$, is bounded between zero and $s^{\max}$. The product demand in time period $t$ is denoted by $d_t$.

We assume that the production amount in each operating mode $i$ is given by a linear function $p = a_i + \bm{b}_i^{\top} \bm{u}$. For some modes, the model parameters $a_i$ and $\bm{b}_i$ are unknown at the beginning of the scheduling horizon, and they can only be estimated by operating in those modes and observing the resulting production amounts. This is captured in the following decision-dependent uncertainty set:
\begin{equation*}\small
\mathcal{P}_t(\bm{x}_{[t]}) = \left\lbrace
\bm{p}_{[t]} \in \mathbb{R}_+^{K_t}:
\begin{array}{l}
  \exists \, a_i, \, \bm{b}_i \; \forall \, i \in \mathcal{I} \quad \text{such that} \\ [4pt]
  0 \leq p_{t'ij} \leq p^{\max}_i x_{t'ij} \quad \forall \, t'=1,\dots,t, \, i \in \mathcal{I}, \, j \in \mathcal{J}_i \\[4pt]
  a_i + \bm{b}_i^{\top} \bm{u}_{ij} - p_i^{\max} (1 - x_{t'ij}) \leq p_{t'ij} \leq a_i + \bm{b}_i^{\top} \bm{u}_{ij} \quad \forall \, t'=1,\dots,t, \, i \in \mathcal{I}, \, j \in \mathcal{J}_i \\[4pt]
  a^{\min}_i \leq a_i \leq a^{\max}_i \quad \forall \, i \in \mathcal{I} \\ [4pt]
  \bm{b}^{\min}_i \leq \bm{b}_i \leq \bm{b}^{\max}_i \quad \forall \, i \in \mathcal{I}
\end{array}
\right\rbrace
\end{equation*}
with $K_t = t \sum_{i \in \mathcal{I}} |\mathcal{J}_i|$. Only if $x_{tij}=1$, the production amount $p_{tij}$ can only take a nonzero value, which depends on the chosen operating point. Each operating point that has been selected up to time period $t$ activates a linear equation, which can be interpreted as a cut that reduces the size of the uncertainty set. As a result, the linear model of an operating mode is learned through the accumulation of observations, which are the production amounts at different operating points. If $\bm{u}_{ij} \in \mathbb{R}^n$, and all model parameters of mode $i$ are unknown, one needs to operate in mode $i$ at different operating points at least $n+1$ times in order to uniquely determine $a_i$ and $\bm{b}_i$.

We use a small instance to illustrate the effect of active learning in this problem. The data are as follows: $T = 4$, $c = 5$, $\delta = 2$, $\bm{d} = (32.5, 15, 51, 42.5)$, $s^{\max} = 100$ and $s_0 = 0$. We consider two operating modes and two possible operating points per mode, with each operating point $j \in \mathcal{J}_i$ defined by a scalar input $u_{ij}$. We have $u_{11} = 10$, $u_{12} = 46$, $u_{21} = 25$, and $u_{22} = 44$. The values of the parameters describing the uncertainty set are shown in Table \ref{tab:Example4Parameters}. Here, we assume that the plant has not been operated in any of the two modes prior to the beginning of the scheduling horizon.

\begin{table}[ht!]\centering
\setlength\tabcolsep{10pt}	
	\caption{Parameters for the description of the uncertainty set in the production scheduling with active parameter estimation problem.}
	\begin{tabular}{cccccc}
	\toprule
     & $p^{\max}_i$ & $a^{\min}_i$  & $a^{\max}_i$  & $b^{\min}_i$ & $b^{\max}_i$   \\
    \midrule
    \textbf{Mode 1} & 65.2 & 5  & 10  & 1    & 1.2  \\
    \textbf{Mode 2} & 56.2 & 20 & 21  & 0.5  & 0.8  \\
 	\bottomrule
	\end{tabular}
	\label{tab:Example4Parameters}
\end{table}

The optimal adjustable production schedule is shown in Figure \ref{fig:Example4-GanttChart}. The solution suggests to first operate at point 1 of mode 2 and choose the subsequent schedule depending on the realization of $p_{121}$. The schedule at the top is selected if $p_{121} \geq 36.75$, and the bottom schedule is selected if $p_{121} < 36.75$. One can see that in both schedules, the operating points selected in time periods 2 and 3 are the same; the schedules only differ in the operating mode chosen for time period 4. By only observing the production amount at operating point 1 of mode 2, we cannot fully infer the model for mode 2, i.e. determine $a_2$ and $b_2$. However, the observation reduces the uncertainty for mode 2, which allows us to determine in which mode the plant should operate in time period 4.

\begin{figure}[ht!]\centering
\includegraphics[width=3in]{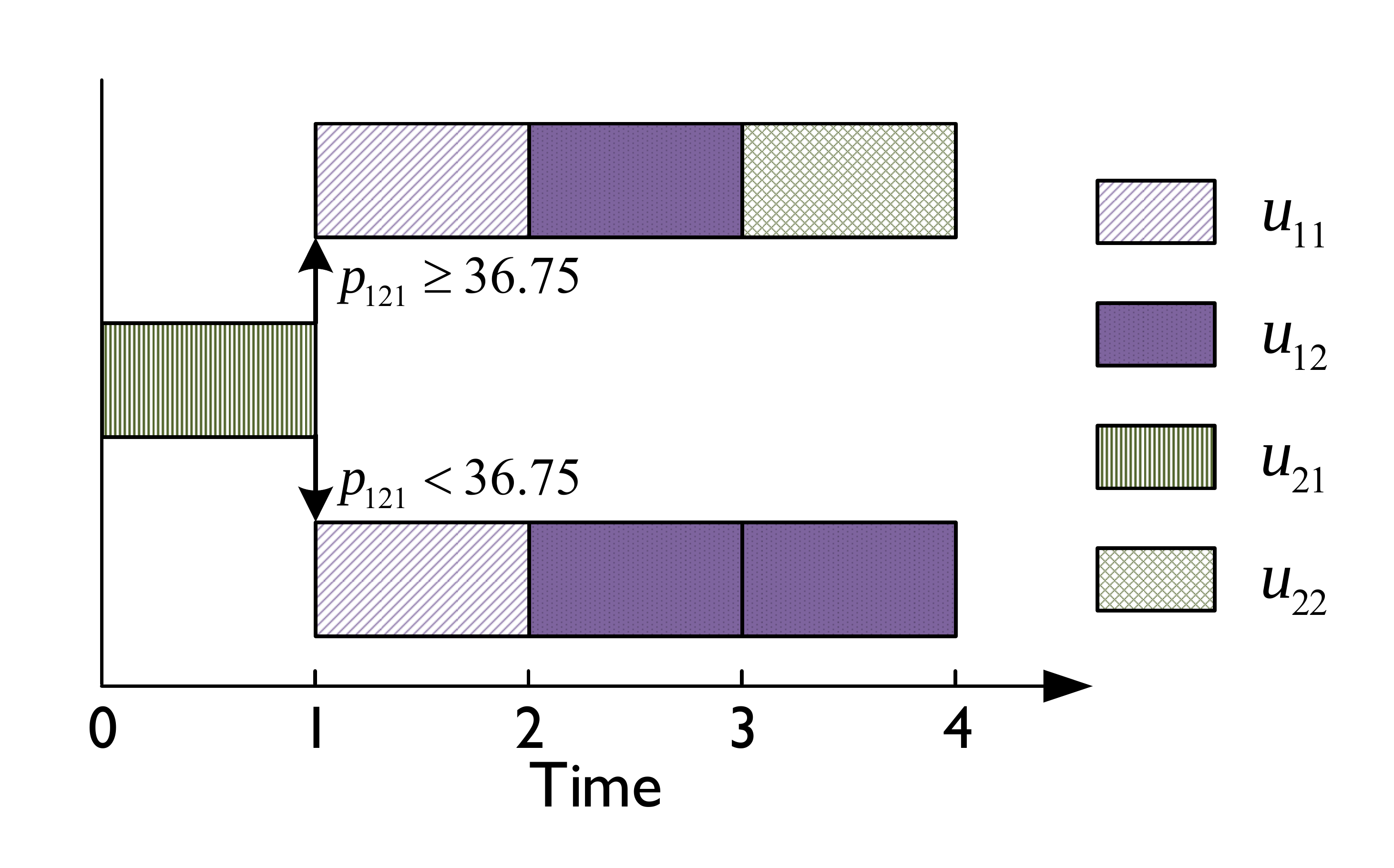}
\caption{Optimal adjustable production schedule with active parameter estimation. The solution indicates that recourse is only required once.}
\label{fig:Example4-GanttChart}
\end{figure}

The effect of learning is further illustrated in Figure \ref{fig:Example4-Case1}, which shows how the uncertainty set is reduced over time in a particular scenario with $p_{121} > 36.75$, in which case the schedule at the top of Figure \ref{fig:Example4-GanttChart} is applied. Each of the four subfigures shows the production amount as a function of the input for a specific time period and the mode chosen for that time period. The black solid vertical lines indicate the marginal uncertainty sets of the production amounts at the available operating points before any observations are made, as given by the ranges of possible values for $a_i$ and $b_i$. In this particular scenario, a production amount of 39 is observed in time period 1 when the plant operates at point 1 of mode 2. Given this observation and $a_2 \in [20,21]$, the uncertainty with respect to $b_2$ is immediately reduced such that the line representing the linear model of mode 2 has to lie within the cyan shaded area shown in Figure \ref{fig:Example4-Case1-Period1}. One can see that this results in a substantial reduction of uncertainty with respect to the unobserved operating point 2 of mode 2. The same effect is seen in time period 2 when the plant operates at point 1 of mode 1 (Figure \ref{fig:Example4-Case1-Period2}). After operating at point 2 of mode 1 in time period 3, $a_1$ and $b_1$ can be uniquely determined given the observations $p_{211}$ and $p_{312}$. The linear model of mode 1 is represented by the red solid line in Figure \ref{fig:Example4-Case1-Period3}. Finally, the linear model of mode 2 will also be exactly known after operating at point 2 of mode 2 in time period 4, as indicated in Figure \ref{fig:Example4-Case1-Period4}.

\begin{figure}[ht]\centering
\subfloat[Period 1 - mode 2]{
  \label{fig:Example4-Case1-Period1}
	\fbox{\includegraphics[height=2.5in]{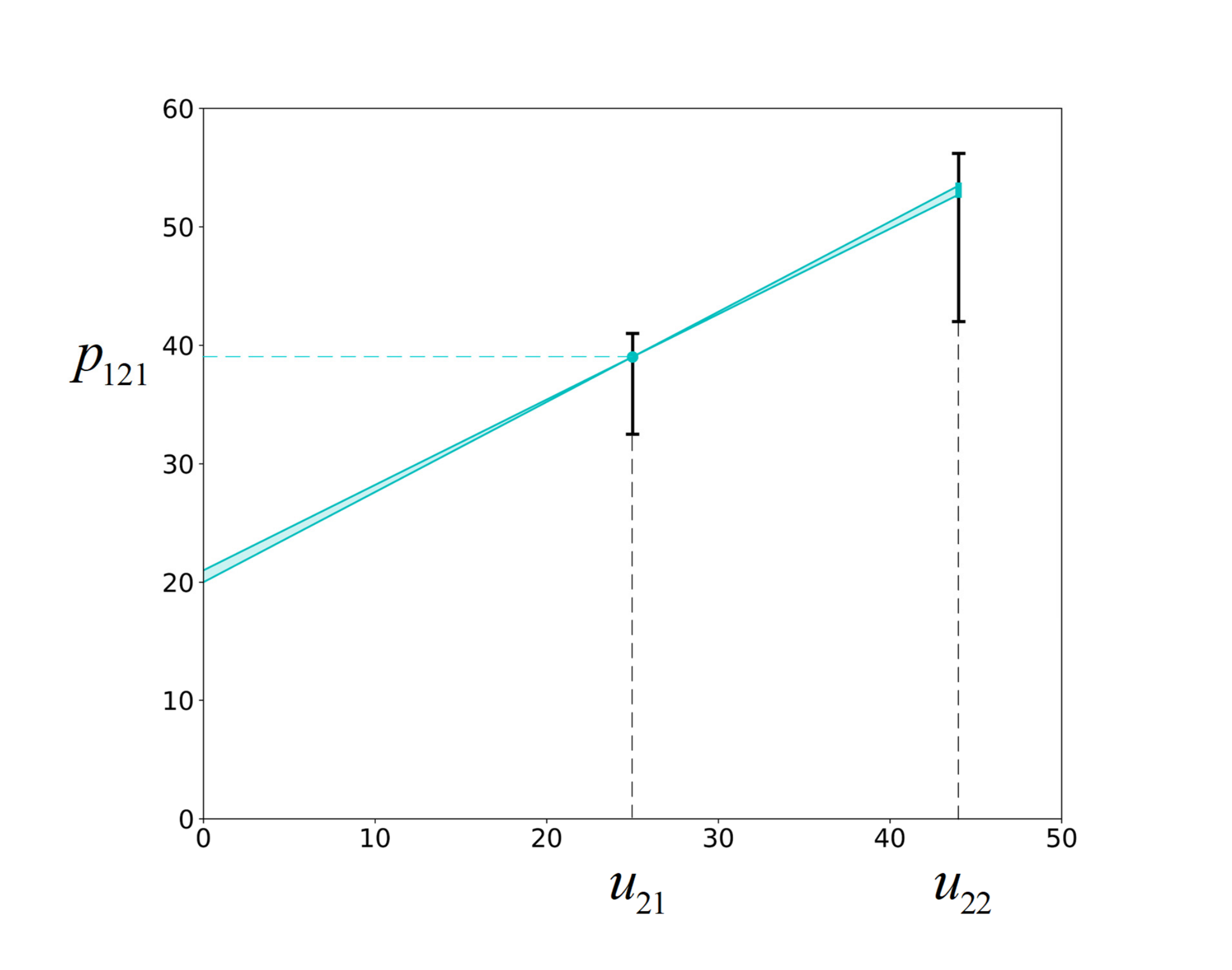}}} 
\subfloat[Period 2 - mode 1]{
  \label{fig:Example4-Case1-Period2}
	\fbox{\includegraphics[height=2.5in]{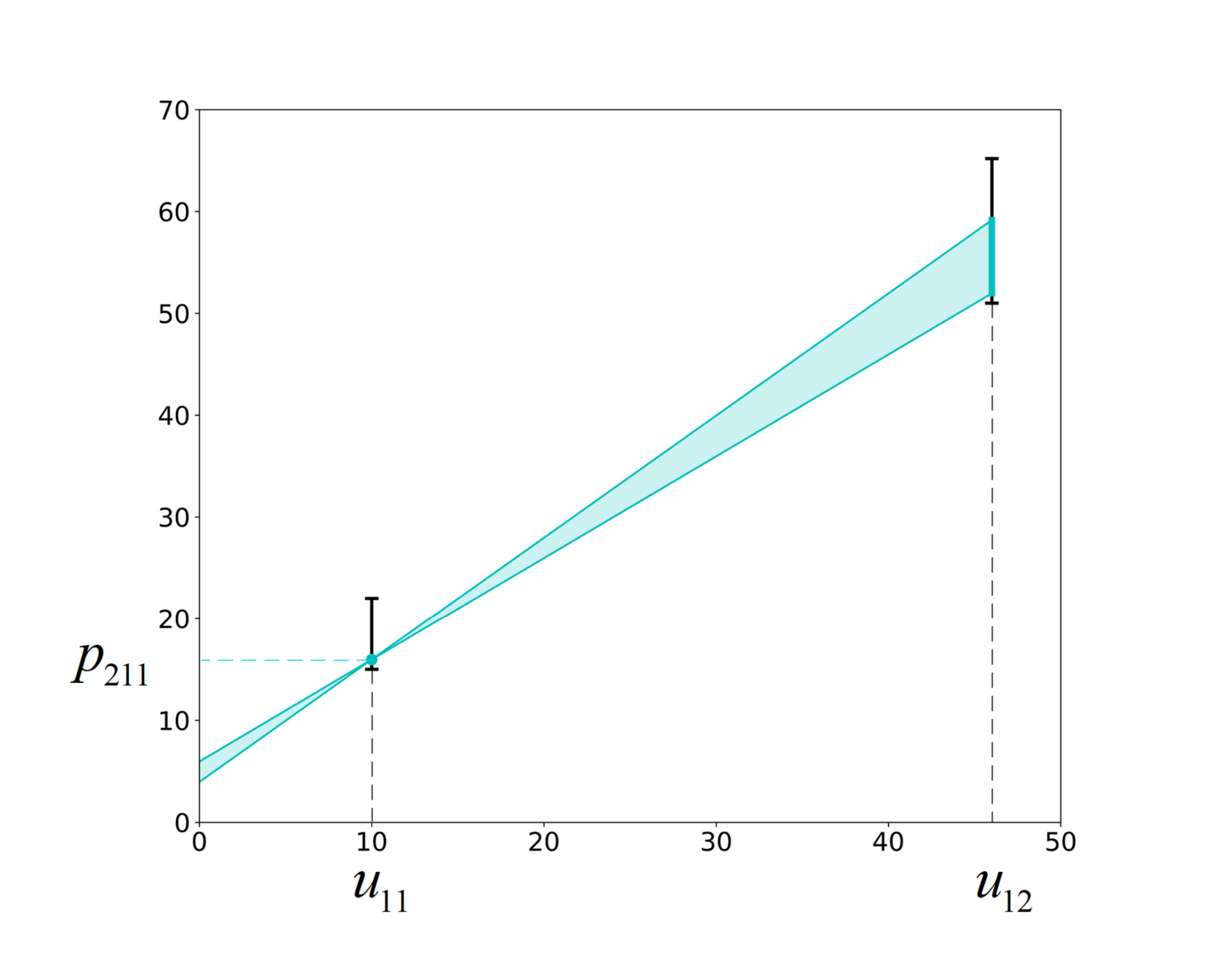}}} \\
\subfloat[Period 3 - mode 1]{
  \label{fig:Example4-Case1-Period3}
	\fbox{\includegraphics[height=2.5in]{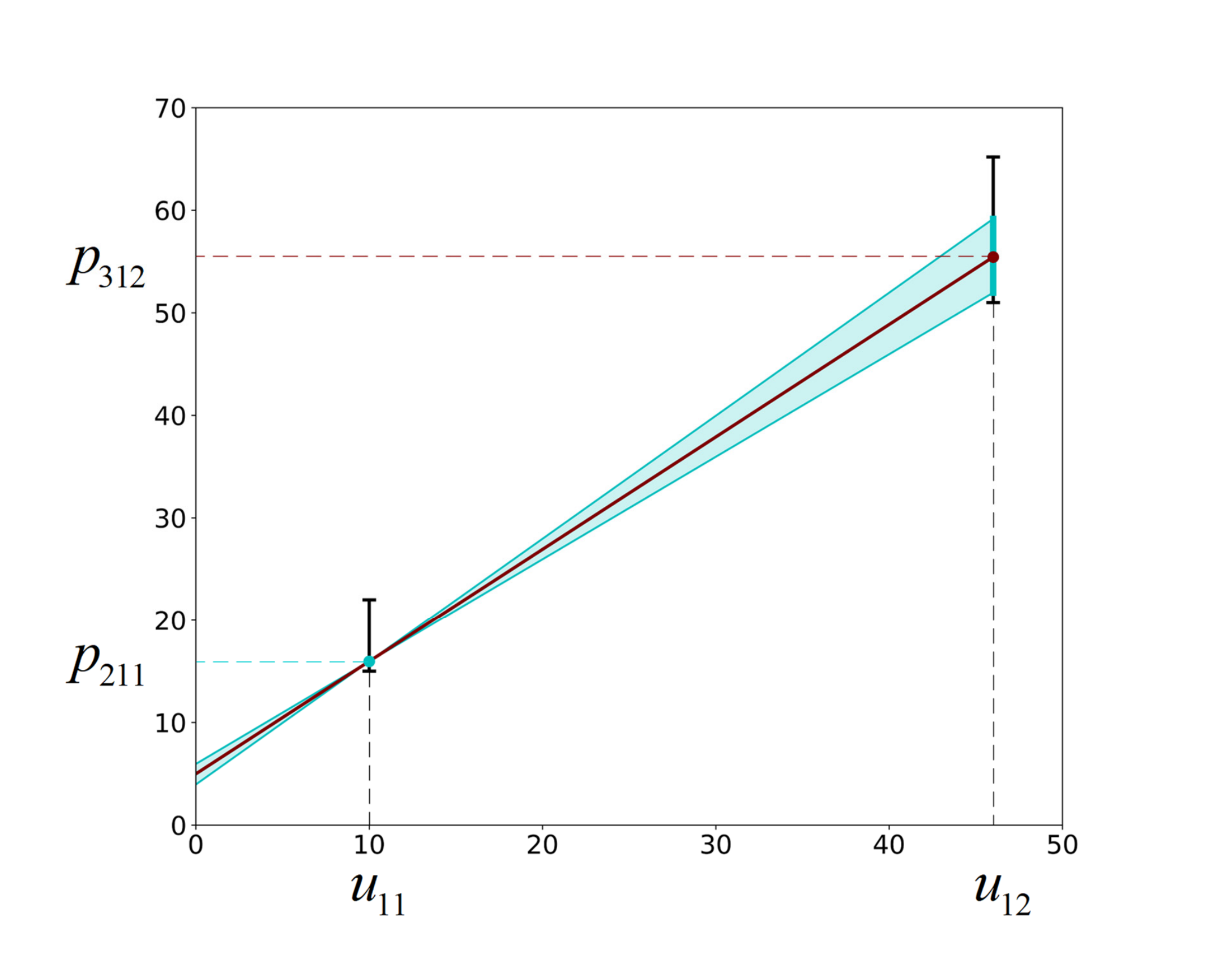}}} 
\subfloat[Period 4 - mode 2]{
  \label{fig:Example4-Case1-Period4}
	\fbox{\includegraphics[height=2.5in]{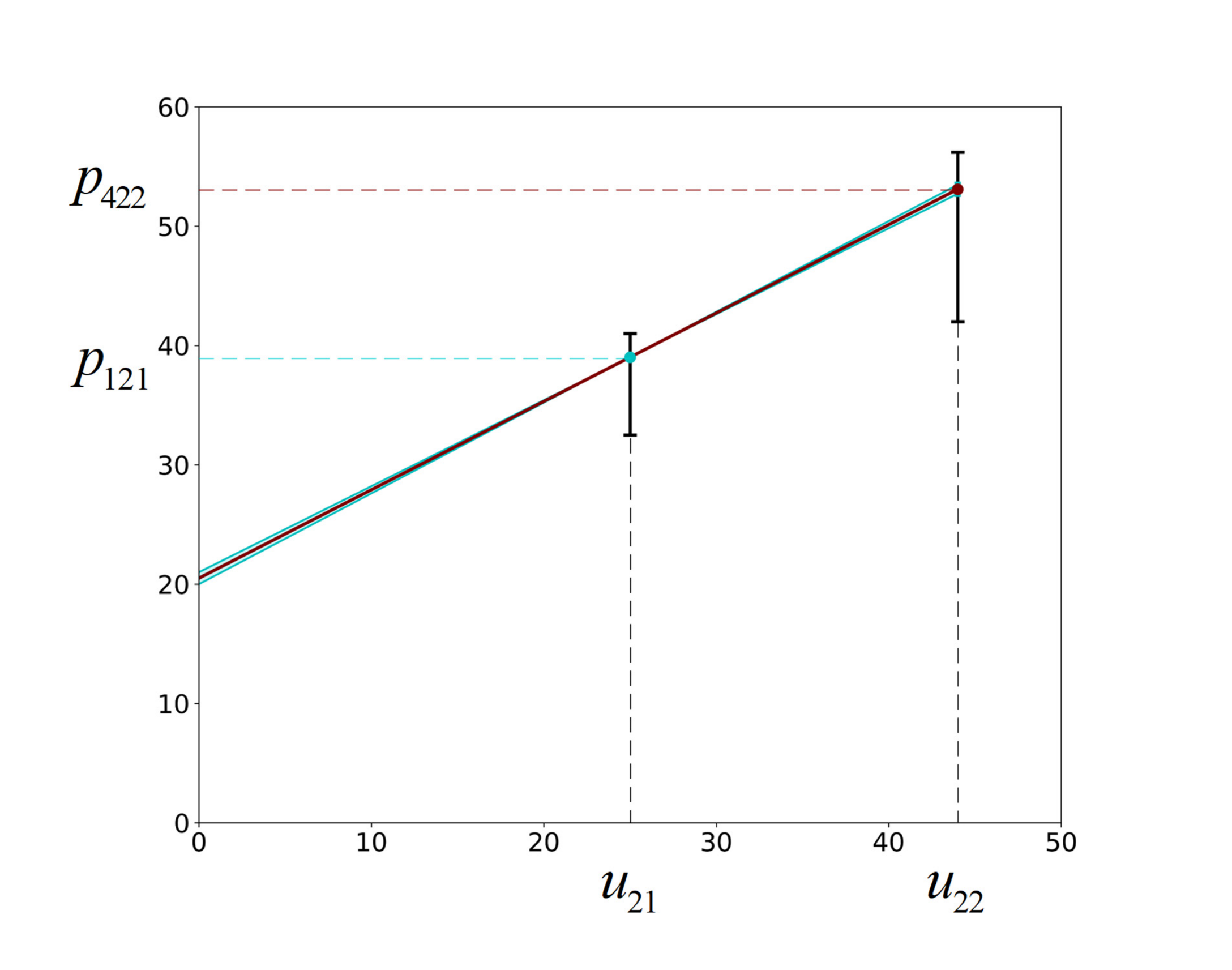}}}	
\caption{Reduction of uncertainty by observing materialized production amounts over time for one specific scenario (defined by one specific set of observations).}
\label{fig:Example4-Case1}
\end{figure}

The optimal value of the problem is 819.2. Despite the simplicity of the illustrative example, the problem is computationally challenging. This is in part due to the large number of variables and constraints resulting from the reformulation, but even more due to the weak relaxation of the MILP. The latter is indicated by a lower bound that only improves very slowly in the branch-and-bound algorithm, despite the incumbent being the optimal solution already. Computational strategies for mitigating these challenges will be explored and discussed in future work.

\section{Conclusions}
\label{sec:Conclusions}

In this work, we have developed an adjustable robust optimization framework that can simultaneously consider three types of endogenous uncertainty: type 1, type 2a, and type 2b, as defined in a new refined classification of endogenous uncertainty that distinguishes between the decision-dependent \textit{materialization} and \textit{observation} of uncertain parameters. One compelling insight from this extended view is that active learning, a sequential form of experimental design, can be formulated as an optimization problem under endogenous uncertainty. As a result, the proposed approach provides a means of performing integrated optimization and active learning in a sequential decision-making environment.

Our framework considers decision-dependent polyhedral uncertainty sets, and we have applied a decision rule approach based on the concept of lifted uncertainty that incorporates continuous and binary recourse, including recourse decisions that affect the uncertainty set. Multistage decision-dependent nonanticipativity is modeled using auxiliary uncertain parameters, which preserves the tractability of the decision rule approach. This eventually results in a mixed-integer optimization problem that can be directly solved using off-the-shelf solvers.

The proposed methodology significantly expands our capability to model and solve data-driven optimization problems under uncertainty, particularly those involving active learning. Computational challenges still exist, which we plan to address in our future work. Nonetheless, this approach promises to be especially suited for applications that involve complex constraints, expensive exploration actions, and strong links between different decision stages. To highlight the relevance of these kinds of problems, we have conducted computational case studies covering a variety of applications. The results demonstrate the versatility of the proposed modeling framework as well as the importance of endogenous uncertainty. As these aspects become increasingly relevant in PSE applications, we hope that this work can contribute toward the development of more comprehensive and efficient methods for such problems.

\section*{Acknowledgments}
Wei Feng gratefully acknowledges the financial support from the China Scholarship Council (CSC) (No. 201906320317).


\bibliographystyle{unsrtnat}
\bibliography{library}

\end{document}